\documentclass[12pt]{amsart}

\usepackage{geometry}                
\usepackage{amsmath, amssymb, amsthm}
\usepackage{tensor}
\usepackage{graphicx,color}
\usepackage{paralist}
\usepackage{mathrsfs}
\usepackage{tikz}
\usepackage{extarrows}
\usepackage[T1]{fontenc} 
\usepackage{bm}
\usepackage{bbm}
\usepackage{paralist}
\usepackage{dsfont}
\usepackage{mathtools}
\usepackage{mathrsfs}
\usepackage{caption}
\usepackage{subcaption}
\usepackage[shortlabels]{enumitem}
\usepackage{pdflscape}
\usepackage{kbordermatrix}
\kbalignrighttrue

\usepackage[colorlinks,pdftex]{hyperref} 
	\hypersetup{
	    unicode=false,          
	    pdftoolbar=true,        
	    pdfmenubar=true,        
	    pdffitwindow=false,     
	    pdfstartview={FitH},    
	    pdftitle={},    
	    pdfauthor={Aaron Dall},     
	    pdfnewwindow=true,      
	    colorlinks=true,       
	    linkcolor=magenta,          
	    citecolor=blue,        
	    filecolor=magenta,      
	    urlcolor=black,           
	    citebordercolor={0 1 1},
	    linkbordercolor={0 0 1}
		}

\makeatletter
\newtheorem*{rep@theorem}{\rep@title}
\newcommand{\newreptheorem}[2]{%
\newenvironment{rep#1}[1]{%
 \def\rep@title{#2 \ref{##1}}%
 \begin{rep@theorem}}%
 {\end{rep@theorem}}}
\makeatother

\newtheorem{theorem}{Theorem}[section]
\newtheorem*{theorem*}{Theorem}
\newtheorem{proposition}[theorem]{Proposition} 
\newtheorem{lemma}[theorem]{Lemma} 
\newtheorem{corollary}[theorem]{Corollary}
\newtheorem{conjecture}[theorem]{Conjecture}

\theoremstyle{remark}

\newtheorem{example}{Example}
\newtheorem*{example*}{Running Example}
\newtheorem*{example2*}{Example}
\newreptheorem{theorem}{Theorem}

\theoremstyle{definition}
\newtheorem{definition}{Definition}

		\makeatletter
		\@tfor\next:=abcdefghijklmnopqrstuvwxyzABCDEFGHIJKLMNOPQRSTUVWXYZ\do{%
		  \def\command@factory#1{%
		    \expandafter\def\csname bf#1\endcsname{\mathbf{#1}}
		  }
		 \expandafter\command@factory\next
		}

		\makeatletter
		\@tfor\next:=ABCDEFGHIJKLMNOPQRSTUVWXYZabcdefghijklmnopqrstuvwxyz\do{%
		  \def\command@factory#1{%
		    \expandafter\def\csname bb#1\endcsname{\mathbbm{#1}}
		  }
		 \expandafter\command@factory\next
		}

		\makeatletter
		\@tfor\next:=ABCDEFGHIJKLMNOPQRSTUVWXYZ\do{%
		  \def\command@factory#1{%
		    \expandafter\def\csname c#1\endcsname{\mathcal{#1}}
		  }
		 \expandafter\command@factory\next
		}

		\makeatletter
		\@tfor\next:=ABCDEFGHIJKLMNOPQRSTUVWXYZ\do{%
		  \def\command@factory#1{%
		    \expandafter\def\csname scr#1\endcsname{\mathscr{#1}}
		  }
		 \expandafter\command@factory\next
		}		

		\makeatletter
		\@tfor\next:=ABCDEFGHIJKLMNOPQRSTUVWXYZ\do{%
		  \def\command@factory#1{%
		    \expandafter\def\csname r#1\endcsname{\textsf{\upshape{#1}}}
		  }
		 \expandafter\command@factory\next
		}				

	\newcommand\defn[1]{\textbf{\emph{#1}}} 
	\newcommand{\Iff}{if and only if }

	\newcommand{\wrt}{with respect to }

	\newcommand{\disjointUnion}{\bigsqcup}
	\newcommand{\smallDisjointUnion}{\sqcup}

	\newcommand{\isomorphic}{\cong}

	\newcommand{\activeElements}[2]{\Act_{#1}({#2})} 
	
	\newcommand{\bases}{\cB}
	\newcommand{\circuits}{\cC}
	\newcommand{\circuit}{C}
	\newcommand{\cocircuits}{\circuits^{*}}
	\newcommand{\cocircuit}{\circuit^{*}}

	\newcommand{\dualMatroid}{\cM^{*}}
	
	\newcommand{\independents}{\cI}
	\newcommand{\intActive}{\rI\rA}
	\newcommand{\intPassive}{\rI\rP}
	\newcommand{\intPerfectMatroids}{\cI\cP}
	\newcommand{\matroidComplex}{\Delta(\matroid)}
	\newcommand{\matroidDual}{\matroid^{*}}
	\newcommand{\matroid}{\cM}	
	
	\newcommand{\STA}{\textit{STA}}
	
	\newcommand{\uniformMatroid}[2]{\cU_{#1,#2}}

	\newcommand{\intOrder}{\preccurlyeq_{\Int}}
	\newcommand{\isCovered}{\triangleleft}
	
	\newcommand{\lessLex}{\prec_{\lex}}
	\newcommand{\poset}{\rP}

	\newcommand{\meet}{\wedge}
	\newcommand{\join}{\vee}

\renewcommand{\kbldelim}{(}
\renewcommand{\kbrdelim}{)}
\renewcommand{\kbrowstyle}{\displaystyle}
\renewcommand{\kbcolstyle}{\displaystyle}

\DeclareMathOperator{\Act}{Act}

\DeclareMathOperator{\dom}{dom}

\DeclareMathOperator{\height}{ht}

\DeclareMathOperator{\IA}{IA}

\DeclareMathOperator{\Int}{int}
\DeclareMathOperator{\IP}{IP}
\DeclareMathOperator{\lex}{lex}

\DeclareMathOperator{\minBasis}{MinBas}
\DeclareMathOperator{\PE}{PE}
\DeclareMathOperator{\PL}{PL}

\DeclareMathOperator{\rank}{rank}

		\makeatletter
		\def\@BracContents{} 
		\newcommand{\@BracKern}{\kern-\nulldelimiterspace}
		\newcommand{\suchthat}{%
		   \nonscript\;
		   \ifnum\currentgrouptype=16
		     \middle|
		   \else
		     \@suchthat
		   \fi
		   \nonscript\;} 
		\newcommand{\@suchthat}{%
		   { 
		   \let\suchthat\@empty 
		   \left.\@BracKern 
		   \vphantom{\@BracContents} 
		   \middle| 
		   \right.\@BracKern 
		   } 
		 }
		\makeatother


\title{Internally Perfect Matroids}

\author{Aaron Dall}

\thanks{This work was mostly completed during the author's doctoral studies at the Universitat Polit\`ecnica de Catalunya and was partially supported by the project MINECO MTM2012-30951/FEDER with additional support coming from the MCINN grant BES-2010-030080. The author wishes to thank Julian Pfeifle for many illuminating discussions and Raman Sanyal for helpful suggestions regarding Section \ref{section:internallyPerfectMatroids}.}

\email{aaronmdall@gmail.com}



\date{}          

\begin{document}
\begin{abstract}
In 1977 Stanley proved that the $h$-vector of a matroid is an $\mathcal{O}$-sequence and conjectured that it is a pure $\mathcal{O}$-sequence. 
In the subsequent years the validity of this conjecture has been shown for a variety of classes of matroids, though the general case is still open. 
In this paper we use Las Vergnas' internal order  to introduce a new class of matroids which we call internally perfect.
We prove that these matroids satisfy Stanley's Conjecture and compare them to other classes of matroids for which the conjecture is known to hold. 
We also prove that, up to a certain restriction on deletions, every minor of an internally perfect ordered matroid is internally perfect.
\end{abstract}

\maketitle


\section{Introduction}
\label{section:introduction}

An ordered matroid is a matroid $\matroid = (E,\bases)$ together with a linear ordering of the ground set $E$.
Given an ordered matroid $\matroid$ and a basis $B \in \bases(\matroid)$, an element $e \in B$ is internally passive (with respect to $B$) if there is a basis $B' \in \bases$ such that $B' = B - \{e\} \cup \{e'\}$ where $e'<e$ in the linear ordering of the ground set.
An element that is not internally passive is internally active.

The internal order of an ordered matroid is the poset~$\poset_{\Int}(\matroid) = (\bases \cup \hat{1}, \intOrder)$ on the set of bases of $\matroid$ together with an artifical maximal element $\hat{1}$, where~$B \intOrder B'$ \Iff every internally passive element of $B$ is internally passive in $B'$.
In~\cite{lasVergnas2001active}, Las Vergnas proves that the internal order of an ordered matroid is a graded lattice and that the height of a basis $B$ in~$\poset_{\intOrder}(\matroid)$ is given by the number of internally passive elements of $B$.
The unique minimal element of $\poset_{\intOrder}(\matroid)$ is the lexicographically smallest basis of $\matroid$ and is denoted $B_{0}(\matroid)$.

Given a basis $B$ of $\matroid$, the~$\STA$-decomposition of $B$ is the partition of $B$ into sets $S,T,$ and $A$ where 
	\begin{align*}
		S &= S(B) \text{ is the set of internally passive elements of $B$ not in $B_{0}(\matroid)$}, 
		\\T&=T(B) \text{ is the set of internally passive elements of $B$ in $B_{0}(\matroid)$, and} 
		\\A&=A(B) \text{ is the set of internally active elements.}
	\end{align*}
We typically express the $\STA$-decomposition of $B$ in the form $B=S^{T}_{A}$.
For each~$f \in S$, write $B_{f}$ for the lexicographically smallest basis in $\matroid$ containing~$f \cup T$ and write $T(B;f)$ for $T(B_{f})$.

Let $B = S^{T}_{A}$ be a basis of an ordered matroid $\matroid$.
Then  $B$ is  $\defn{internally deficient}$ if $T \neq \bigcup_{S}T(B;f)$.
If $T = \bigcup_{S}T(B;f)$ but the union is not disjoint, then $B$ is said to be \defn{internally abundant}.
If $T = \disjointUnion_{S}T(B;f)$ is a disjoint union, then $B$ is called \defn{internally perfect}.
The ordered matroid $\matroid$ is \defn{internally perfect} if every basis of $\matroid$ is internally perfect.

(We use the modifier ``internally'' to stress that our definitions depend on the notion of internal activity as opposed to the dual notion of external activity.
As we will have no cause to mention the external activity again in this paper, we typically drop the modifier when the additional stress is superfluous.)

Though the definitions of perfect, abundant, and deficient bases as stated above are useful in computations, they each have a more intuitive characterization in terms of the join operator of $\poset_{\Int}(\matroid)$; see Proposition \ref{proposition:joinsOfProjections}.
Other highlights of Section~\ref{section:perfectAbundantDeficient} include proofs of the existence of perfect bases in any ordered matroid and that all rank $2$ matroids are perfect (see Proposition \ref{proposition:perfectBasesExist}).
These preliminary results give way in later sections to our two main structural results concerning perfect matroids.

The first central result details to what extent minors of a perfect matroid $\matroid$ are perfect when the ordering of the ground set of each minor is induced from that of $\matroid$. 

	\newtheorem*{theorem:minorClosedIfDelOffB0}{Theorem \ref{theorem:minorClosedIfDelOffB0}}
	\begin{theorem:minorClosedIfDelOffB0}
	\emph{Let $\matroid = (E, \bases, \phi)$ be an internally perfect ordered matroid with initial basis $B_{0}$ and let $F_{1}$ and $F_{2}$ be disjoint subsets of $E$ such that any element of~$F_{2} \cap B_{0}$ is a coloop.
		Then the minor $\matroid /F_{1}\setminus F_{2}$ is internally perfect with respect to the ordering of its ground set induced by the order of $\matroid$.}
	\end{theorem:minorClosedIfDelOffB0}

The second of our main results pertains to a conjecture of R. Stanley concerning the structure of the $h$-vectors of matroids.
Indeed, this conjecture provided the original motivation for our study of internally perfect matroids.
As such we take a momentary diversion to sketch the background of the conjecture as well as the work done in recent years toward finding a proof.

The $h$-vector, $h(\matroid) = (h_{0}, h_{1}, \dots, h_{r})$, of a rank $r$ matroid $\matroid$ is defined to be the $h$-vector of the independence complex of $\matroid$, that is, the $h$-vector of the simplicial complex on $E$ consisting of all subsets $I \subseteq B$ for some $B \in \bases$.
Given such a matroid $\matroid$, the entry $h_{i}$ can be computed in a number of ways.
For example, it is the coefficient on $x^{r-i}$ in the evaluation of the Tutte polynomial,~$T_{\matroid}(x,y)$, of~$\matroid$ at $y=1$; see~\cite{brylawski1992tutte}.
Equivalently, for any linear ordering of $E$, the entry $h_{i}$ is the number of bases of $\matroid$ with $r-i$ internally active elements; see \cite{bjorner40homology}.
By a result of Las Vergnas, one can also obtain $h_{i}$ by counting the number of bases at height $i$ in $\poset_{\Int}(\matroid)$; see Theorem \ref{theorem:internalOrderFacts}.
Any of these results implies that $h_{0}=1$ and $h_{i}\ge 0$ for all $i \in [r]$.

While a result of Kruskal \cite{kruskal1963number} and Katona \cite{katona1987theorem} gives an explicit description of the possible vectors that can occur as the $h$-vector of simplicial complexes, no such description is yet known when restricting to the matroid complex case.
So when given a vector $\bfv \in \bbN^{r+1}$ it is natural to ask for tests that would verify (or rebuff)~$\bfv$ as the $h$-vector of a rank $r$ matroid.

In \cite{stanley1977cohen}, Stanley gives such a test by considering order ideals inside of the poset~$\poset = (\bbN^{r+1},\preccurlyeq_{\dom})$, for some $r\in \bbN$, ordered by $\bfv \preccurlyeq \bfw$ \Iff~$\bfv_{i} \le \bfw_{i}$ for all~$i\in [r+1]$.
An order ideal $\cO$ of $\poset$ is the downset in $\poset$ of a subset~$V \in \bbN^{r+1}$ and the~$\cO$-sequence $(\cO_{0}, \cO_{1}\dots,\cO_{r+1})$ of $\cO$ is the vector encoding the number of elements in $\cO$ with coordinate sum equal to $i$ ($i\in \{0,1,\dots, r+1\}$).
An $\cO$-sequence is called pure if it is the $\cO$-sequence of a pure order ideal, that is, of an order ideal whose maximal elements all have the same coordinate sum.
Stanley proved that the $h$-vector of any rank $r$ matroid is the $\cO$-sequence of an order ideal in $\poset$  and made the following conjecture:
	\begin{conjecture}[Stanley 1977]
	\label{conjecture:stanley}
			The $h$-vector of a matroid is a pure $\cO$-sequence.
	\end{conjecture}

Following a 23-year period during which no partial results were published, the last fifteen years have seen a flurry of research concerning Stanley's Conjecture~\cite{constantinescu2012generic,constantinescu2012h,deloera2011h,klee2014lexicographic,lopez1997chip,merino2012structure,oh2013generalized,schweig2010h}.
These results are essentially of two types.
In each of~\cite{deloera2011h,lopez1997chip,merino2012structure,oh2013generalized,schweig2010h} a certain class of matroids is considered and Stanley's Conjecture is shown to hold by exploiting properties of the class.
In particular, Stanley's conjecture is known to hold for cographic matroids (Merino~\cite{lopez1997chip}), paving matroids~(De Loera et al.~\cite{deloera2011h}) and matroids with rank no more than four (Klee--Samper~\cite{klee2014lexicographic}).

In each of the other papers referenced above, general properties of either matroids or order ideals are studied and then used to prove Stanley's Conjecture for a particular class of matroids.
The second of our main results is of this type.
It states that, up to a relabeling of the nodes, the internal order of an internally perfect matroid is a pure order ideal.
	\newtheorem*{theorem:perfectImpliesStanley}{Theorem \ref{theorem:perfectImpliesStanley}}
	\begin{theorem:perfectImpliesStanley}
		The internal order of an internally perfect matroid is isomorphic to a pure order ideal.
	\end{theorem:perfectImpliesStanley}
In particular this result implies that if $\matroid=(E,\bases)$ is a matroid such that there exists a linear ordering of the $E$ making $\matroid$ an internally perfect matroid, then~$\matroid$ satisfies Stanley's Conjecture.

This article is organized as follows.
In Section~\ref{section:preliminaries} we fix notation and give the necessary background on matroids and the internal order.
In Section~\ref{section:perfectAbundantDeficient} we define and prove preliminary properties of perfect, abundant, and deficient bases of an order matroid.
We turn to the proof of Theorem \ref{theorem:minorClosedIfDelOffB0} concerning the minors of perfect matroids in Section \ref{section:minors}.
We also provide an example showing that the theorem is the best possible result when the linear order on the ground set of the minor is induced from that of the original matroid.
Moreover, we conjecture that one can always find a reordering of the ground set of a minor of perfect matroid such that the minor is perfect with respect to the new order.
In Section \ref{section:internallyPerfectMatroids} we prove Theorem \ref{theorem:perfectImpliesStanley} by explicitly giving the poset isomorphism between the internal order of a perfect matroid and a pure order ideal.
Finally, in Section \ref{section:examples} we construct a variety of examples of internally perfect matroids including an infinite family of cographic matroids as well as an example of an internally perfect matroid that is not contained in any of the classes for which Stanley's Conjecture is known to hold.


\section{Preliminaries: Matroids and the Internal Order}
\label{section:preliminaries}

	First we fix some notation.
	For a positive integer $n$ we write $[n]:=\{1,2,\dots, n\}$.
	We use standard basic set theory notation with the following exception: for a singleton $\{f\}$ we typically suppress the set braces  and simply write $f$.

	We now review some basic matroid terminology following \cite{oxley1992matroid}.
	Then we recall the facts we need concerning activities and the internal order essentially following~\cite{lasVergnas2001active}.
	We assume a basic familiarity with posets and simplicial complexes at the level of Chapter 3 of \cite{stanley2011enumerative} and Chapter 1 of \cite{miller2005combinatorial}, respectively.

	\subsection{Matroids}
	\label{subsection:matroids}

		A matroid $\matroid = (E, \bases)$ is a pair consisting of a finite set $E$ and a set of bases $\bases$ satisfying the following axioms:
			\begin{enumerate}
		        \item $\bases$ is a nonempty set, and
		        \item if $B_{1}$ and $B_{2}$ are in $\bases$ and $e \in B_{1} - B_{2}$, then there is an $f \in B_{2} - B_{1}$ such that $B_{1} - e \cup f$ is in $\bases$.
	      	\end{enumerate}
	    Let $\matroid = (E,\bases)$ be a matroid and let $F \subseteq E$.
	    Then $F$ is an \defn{independent set} of~$\matroid$ if it is a subset of some basis.
	    The set of all independent sets of $\matroid$ is denoted $\cI$.
	    A subset of $E$ that is not independent of $\matroid$ is a \defn{dependent set} and a dependent set that is minimal with respect to inclusion is a \defn{circuit}.
		The set of all circuits of $\matroid$ will be denoted $\circuits = \circuits(\matroid)$.
		A \defn{loop} of $\matroid$ is a circuit consisting of one element.
		We write $\cL(\matroid)$ for the set of all loops.
		If two elements $e,f \in E$ form a two-element circuit then they are said to be \defn{parallel}.		
		A maximal collection of elements of $E$ containing no loops such that the elements are pairwise parallel in~$\matroid$ is called a \defn{parallel class} of $\matroid$.
		
		The \defn{rank} of a subset $S \subseteq E$, denoted $\rank_{\matroid}(S)$, is the cardinality of any maximal independent set of $\matroid$ contained in $S$.
		It is easy to see from the definition that every basis of $\matroid$ has the same rank $r$, called the \defn{rank} of~$\matroid$ and written~$\rank (\matroid)$.
		When the matroid under consideration is clear from the context, we typically drop it from the notation.	

		The \defn{dual}  matroid, $\matroidDual = (E, \bases^{*})$, of $\matroid$ is the matroid whose bases are the complements of bases in $\matroid$.
		The bases of the dual matroid are called \defn{cobases}.
		More generally, we prepend the prefix ``co-'' to any object associated to a matroid to indicate that we are discussing the corresponding dual object.
		In particular, the cocircuits of $\matroid$ are the circuits of $\matroid^{*}$ and will be denoted $\cocircuits$.
		For example, a \defn{coloop} of $\matroid$ is a loop in the dual matroid $\matroidDual$. 
		Equivalently, a coloop is an element of the ground set that is in every basis of $\matroid$.

		Let~$\matroid = (E, \bases)$ be a matroid, $B$ be a basis of $\matroid$, and suppose~$e\in E$ is not an element of $B$.
		Then there is a unique circuit, $\circuit(B;e)$, of $\matroid$ contained in the set $B\cup e$ called the \defn{fundamental circuit} of $B$ with respect to $e$. 
		Similarly, for an element~$f \in B$ the \defn{fundamental cocircuit} of~$B$ with respect to~$f$ is the unique cocircuit~$\cocircuit(B;f)$ of $\matroid$ contained in the set $E \setminus B \cup f$.
		It is a basic fact that for $b \in B$ and~$b' \notin B$ the following are equivalent:
			\begin{enumerate}
				\item the set $B' := B - \{b\} \cup \{b'\}$ is a basis;
				\item $b \in \circuit(B;b')$; and
				\item $b \in \cocircuit(B';b')$.
			\end{enumerate}
		The set $B'$ above is said to be obtained from $B$ by \defn{pivoting}.

		Let $\matroid$ be a rank $r$ matroid.
		As the set $\independents$ of independent sets of $\matroid$ is nonempty and closed under taking subsets,  it forms a simplicial complex $\matroidComplex$ on $E$, called the \defn{matroid} (or \defn{independence}) \defn{complex} of $\matroid$.
		The dimension of the matroid complex $\matroidComplex$ is the rank of $\matroid$.
		The \defn{$f$-vector} and \defn{$h$-vector} of a matroid $\matroid$ are the $f$- and~$h$-vector of its matroid complex, respectively.

		Two useful constructions for creating new matroids are deletion and contraction defined, respectively, as follows.
		Let $T \subseteq E$.
		The \defn{deletion} of $\matroid$ at~$T$, written as~$\matroid \setminus T$, is the matroid whose independent sets are~$I \setminus T$ for each~$I \in \independents$, while the \defn{contraction} of $\matroid$ at $T$ is the matroid defined by~$\matroid / T := (\dualMatroid \setminus T)^{*}$.
		A \defn{minor} of $\matroid$ is any matroid that can be obtained from $\matroid$ by a sequence of deletions and contractions.

		Let $\matroid = (E,\bases)$ be a matroid with $|E|=n$.
		If $\phi: E \to [n]$ is a bijection, then the usual ordering on $[n]$ induces a linear order on $E$ where $e \prec_{\phi} f$ if $\phi(e)<\phi(f)$.
		The matroid $\matroid$ together with such a bijection $\phi$ is called an \defn{ordered matroid}.
		When the map $\phi$ is clear from the context, we write $e < f$ in place of $e \prec_{\phi} f$.

	\subsection{The Internal Order}
	\label{subsection:internalOrder}

		Let $\matroid = (E, \bases, \phi)$ be an ordered matroid, let $F \subseteq E$, and let $e \in E$.
		Then $e$ is $\matroid$-\defn{active} with respect to $F$ if there is a circuit $C \subseteq F \cup e$ of $\matroid$ such that $e = \min C$.
		Notice that this definition depends on the ordering induced by the bijection $\phi$.
		The set of all $\matroid$-active elements with respect to $F$ is denoted $\activeElements{\matroid}{F}$.		
		The element $e$ is \defn{internally active} in $\matroid$ with respect to $F$ if 
			\[e \in \intActive_{\matroid}(F) := \activeElements{\matroid^{*}}{E-F} \cap F.\]
		In other words, $e$ is internally active in $\matroid$ with respect to $F$ if $e \in F$ and there is a cocircuit $C^{*}$ of $\matroid$ contained in $E - F \cup e$ such that $e = \min C^{*}$.
		When the underlying ordered matroid is clear from the context we will simply say that $e$ is internally active in $F$. 
		Any element of $F$ that is not internally active with respect to $F$ is \defn{internally passive}.
		We write $\intPassive_{\matroid}(F)$ for the set of all internally passive elements of $F$.

		In this paper we typically focus on the case when $F = B$ is a basis of $\matroid$.
		Notice that in this case an element $f \in B$ is internally active in $B$ if and only if $f$ is the minimum element of the fundamental cocircuit~$\cocircuit(B;f)$.
		As~$\matroid$ is an ordered matroid, there is a lexicographically-smallest basis which we denote by $B_{0}$ and call the \defn{initial basis} of $\matroid$.
		It is trivial to check that $\intActive(B_{0}) = B_{0}$ and $\intPassive(B_{0})= \emptyset$.
		Now consider an arbitrary basis $B \in \bases$.
		By Corollary \ref{corollary:activeInB0} below, any internally active element of $B$ is also an element of the initial basis $B_{0}$, though the converse need not hold.
		On the other hand, if $e \in B - B_{0}$, then $e$ is internally passive in $B$.
		This motivates us to partition $B$ into three sets defined as follows:
				\begin{align}
				\label{eq:defineS}
				S &= S_{\matroid}(B) := \intPassive_{\matroid}(B) - B_{0};
				\\ \label{eq:defineT} T&=T_{\matroid}(B) := \intPassive_{\matroid}(B) \cap B_{0};
				\\ \label{eq:defineA}A&=A_{\matroid}(B) := \intActive_{\matroid}(B).
				\end{align}
		We call the elements of $S = S(B)$ \defn{perpetually passive} (with respect to $\matroid$ and~$B$), while 
		the elements of $T$ are called  \defn{provisionally passive}.
		The reason for this choice of nomenclature will become obvious in the next paragraph.
		We will often write $S^{T}_{A}$ for $B$ when we want to emphasize this partition of the elements of~$B$.
		The basis $B_{0}$ is the only basis with no internally passive elements.
		Let~$B = S^{T}_{A}$ be a basis.
		Then $B$ is \defn{clean} if it has no provisionally passive elements (i.e., $T=\emptyset$).
		The basis $B$ is called \defn{principal} if $|S|=1$, and is called \defn{$f$-principal} if $S = \{f\}$.

		If $B$ and $ B'$ are bases of $\matroid$ such that $B' = B - b \cup b'$ where $b' = \min C^{*}(B',b')$ is an internally active element of $B'$ and $b \in C^{*}(B',b') - b'$, then $B'$ is said to be obtained from $B$ by \defn{internally active pivoting} and we write $B' \longleftarrow^{b'}_{b} B$.
		So an internally active pivot exchanges an internally passive element of $B$ for an internally active element of $B'$.
		Let $B' \longleftarrow^{b'}_{b} B$ be an internally active pivot.
		It is trivial to see that if $e \in S(B) \cap B'$, then $e \in S(B')$. 
		Moreover, if $e \in S(B) \cap B''$ for any basis $B''$, then $e \in S(B'')$ which justifies our calling such an element perpetually passive.
		On the other hand, if $B'$ is obtained from $B$ by an internally active pivot and~$e \in T(B)\cap B'$, then~$e$ is either in $T(B')$  or it may become active in~$B'$.
		It is in this sense that elements of $T(B)$ are provisionally passive with respect to $B$.  
		Note that the only elements of $B_{0}(\matroid)$ that are not provisionally passive with respect to any basis $B$ are the least element of $E$ (with respect to the linear order on $E$) as well as every coloop of $\matroid$.

		Let $\prec$ be the binary relation on the bases of $\matroid$ defined by $B' \prec B$ \Iff $B'$ is obtained from $B$ by internally active pivoting.
		The relation $\prec$ is trivially irreflexive and asymmetric.
		Let $\intOrder$ be the transitive closure of $\prec$.
		\begin{theorem}[\cite{lasVergnas2001active} Proposition 5.2]
		\label{theorem:internalOrderIsAPoset}
			Let $B,B' \in \bases(\matroid)$ be bases of an ordered matroid $\matroid$.
			Then the following are equivalent:
				\begin{enumerate}[label=\upshape(\roman*)]
					\item $B \intOrder B'$;
					\item $\intPassive(B) \subseteq B'$;
					\item \label{eq:ipContainment} $\intPassive(B) \subseteq \intPassive(B')$;
					\item $B$ is the lexicographically-least basis of $\bases$ containing $B \cap B'$.
				\end{enumerate}
		\end{theorem}
		In particular, Theorem \ref{theorem:internalOrderIsAPoset} implies that the pair $\poset_{\Int}(\matroid):=(\bases, \intOrder)$ is a poset.
		The poset $\widehat{\poset}_{\Int}(\matroid) = (\bases \cup \hat{1}, \intOrder)$ where $\hat{1}$ is an artificial top element is the \defn{internal order} of the ordered matroid $\matroid$.
		As we shall see in the next example, the internal order of an ordered matroid depends on the particular choice of ordering of the ground set.


		\begin{example}
		\label{example:perfectNotPerfect}
		Consider the vector matroid $\matroid = \matroid(M)$ on $[6]$ given by the matrix 
\[M := \kbordermatrix{
        \empty & 1 & 2 & 3 & 4 & 5 & 6\cr
        \empty &
        1&
        0&
        0&
        0&
        1&
        0 \cr
        \empty &
        0&
        1&
        0&
        1&
        {-2}&
        1\cr
        \empty&
        0&
        0&
        1&
        0&
        1&
        1\\}
        .\]
The $h$-vector of $\matroid$ is $(1,3,5,5)$.
The natural ordering on the ground set yields an ordered matroid whose internal order is shown on the left in Figure \ref{figure:poset1}. 
In this internal order we have highlighted the three principal chains in $\poset_{\Int}$: the $5$-principal chain has length $2$, while the $6$- and $4$-principal chains have lengths $1$ and $0$ respectively.

Let~$\matroid(M')$ be the ordered matroid obtained by reordering the columns of $M$ using the order~$\{2,3,1,4,5,6\}$:
  \[M' := \kbordermatrix{
        \empty & 2 & 3 & 1 & 4 & 5 & 6 \cr
        \empty &
        0&
        0&
        1&
        0&
        1&
        0\cr
        \empty &
        1&
        0&
        0&
        1&
        {-2}&
        1\cr
        \empty &
        0&
        1&
        0&
        0&
        1&
        1\\
        }.\]
The internal order of this ordered matroid is shown on the right in Figure \ref{figure:poset1}, where we have highlighted the clean bases in red.
The two posets in Figure \ref{figure:poset1} are not isomorphic as there are two height-$3$ bases in the first poset that cover exactly one element while there is only one height-$3$ basis in the second poset that covers one element. 
Nonetheless, both internal orders have a number of features in common. 
For example, both are graded lattices with~$h_{i}(\matroid)$ bases at height $i$.
\qed

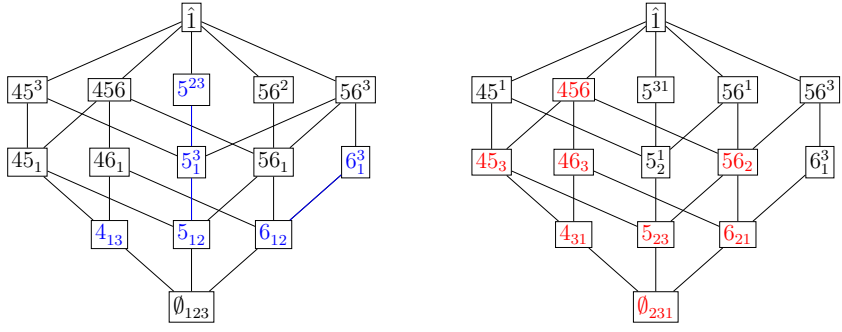
\begin{figure}[htbp]
  \centering 
  \begin{subfigure}[htbp]{0.35\textwidth}
        \centering
        \resizebox{\linewidth}{!}{
        \tikzstyle{every node}=[draw=black,thin, fill = white, rectangle,inner sep=2pt]
  \begin{tikzpicture}[scale=1, vertices/.style={draw, fill=black, circle, inner sep=0pt}]
               \node  (0) at (-0+0,0)		 	{$\emptyset^{}_{123}$};
               \node  (1) at (-1.5+1.5,1.33333)	{$\color{blue}{5}^{}_{12}$};
               \node  (2) at (-1.5+3,1.33333)	{$\color{blue}{6}^{}_{12}$};
               \node  (3) at (-1.5+0,1.33333)	{$\color{blue}{4}^{}_{13}$};
               \node  (4) at (-3+3,2.66667)	 	{$\color{blue}{5}^{3}_{1}$};
               \node  (8) at (-3+4.5,2.66667)	{${56}^{}_{1}$};
               \node  (5) at (-3+6,2.66667)	 	{$\color{blue}{6}^{3}_{1}$};
               \node  (6) at (-3+0,2.66667)		{${45}^{}_{1}$};
               \node  (7) at (-3+1.5,2.66667) 	{${46}^{}_{1}$};
               \node  (9) at (-3+3,4)		 	{$\color{blue}{5}^{23}_{}$};
               \node (10) at (-3+4.5,4)		 	{${56}^{2}$};
               \node (12) at (-3+6,4)		 	{${56}^{3}$};
               \node (11) at (-3+0,4)		 	{${45}^{3}$};
               \node (13) at (-3+1.5,4)		 	{${456}^{}$};
               \node (14) at (-0+0,5.33333)	 	{$\hat{1}$};
       \foreach \to/\from in {0/1, 0/2, 0/3, 1/8, 1/4, 1/6, 2/8, 2/5, 2/7, 3/6, 3/7, 4/12, 4/9, 4/11, 5/12, 6/13, 6/11, 7/13, 8/12, 8/13, 8/10, 9/14, 10/14, 11/14, 12/14, 13/14}
       \draw [-] (\to)--(\from);
       \foreach \to/\from in {1/4, 2/5, 4/9}
       \draw [blue,-] (\to)--(\from);
       \end{tikzpicture}
       }       
    \end{subfigure}
    \quad\quad  
  \begin{subfigure}[htbp]{0.35\textwidth}  
        \centering
        \resizebox{\linewidth}{!}{
        \tikzstyle{every node}=[draw=black,thin, fill = white, rectangle,inner sep=2pt]
        \begin{tikzpicture}[scale=1, vertices/.style={draw, fill=black, circle, inner sep=0pt}]        	
                \node (0) at (-0+0,0)	 		 {$\color{red}\emptyset^{}_{231}$};
                \node (1) at (-1.5+1.5,1.33333)	 {$\color{red}{5}_{23}$};
                \node (2) at (-1.5+3,1.33333)	 {$\color{red}{6}_{21}$};
                \node (3) at (-1.5+0,1.33333)	 {$\color{red}{4}_{31}$};
                \node (4) at (-3+3,2.66667)	 	 {${5}^{1}_{2}$};
                \node (5) at (-3+4.5,2.66667)	 {$\color{red}{56}_{2}$};
                \node (6) at (-3+6,2.66667)	 	 {${6}^{3}_{1}$};
                \node (8) at (-3+1.5,2.66667)	 {$\color{red}{46}_{3}$};
                \node (7) at (-3+0,2.66667)	 	 {$\color{red}{45}_{3}$};
                \node (9) at (-3+3,4)		 	 {${5}^{31}$};
                \node (12) at (-3+4.5,4)		 {${56}^{1}$};
                \node (10) at (-3+6,4)		 	 {${56}^{3}$};
                \node (13) at (-3+1.5,4)		 {$\color{red}{456}$};
                \node (11) at (-3+0,4)			 {${45}^{1}$};
                \node (14) at (-0+0,5.33333)	 {$\hat{1}$};
        \foreach \to/\from in {0/1, 0/2, 0/3, 1/4, 1/5, 1/7, 2/8, 2/5, 2/6, 3/8, 3/7, 4/12, 4/9, 4/11, 5/12, 5/13, 5/10, 6/10, 7/13, 7/11, 8/13, 9/14, 12/14, 10/14, 13/14, 11/14}
        \draw [-] (\to)--(\from);               
        \end{tikzpicture}
           }
	    \end{subfigure}
	    \caption{Non-isomorphic internal orders for the same matroid with different ground set orderings}
		\label{figure:poset1}  
		\end{figure}		
	\end{example}
		
		Next we give a structural result due to Las Vergnas for the internal order of an ordered matroid.	
		For this we need the following definition:
		Given an independent set $I$ of a rank~$r$ ordered matroid $\matroid$, the \defn{minimum basis} containing $I$, written~$\minBasis(I)$, is defined to be the lexicographically-least basis in $\bases$ containing~$I$.	
		\begin{theorem}[\cite{lasVergnas2001active}]
		\label{theorem:internalOrderFacts}
			Let $\matroid = (E,\bases,\phi)$ be an ordered rank-$r$ matroid.
			Then the internal order of $\matroid$, $\widehat{\poset}_{\Int} = (\bases \cup \hat{1}, \intOrder)$, is a graded lattice with height function given by $\height(B) = |\intPassive(B)|$ for all $B \in \bases$.
			The meet and join in $\widehat{\poset}_{\Int}$ of any two bases are given by 
				\begin{align*}
					B_{1} \meet B_{2} &= \minBasis(\intPassive(B_{1}) \cap \intPassive(B_{2}))
					\\ B_{1} \join B_{2} &= 
					\begin{cases}
							\minBasis(\intPassive(B_{1}) \cup \intPassive(B_{2})) \text{ if $\intPassive(B_{1}) \cup \intPassive(B_{2}) \in \cI(\matroid)$ and }
							\\ \hat{1} \text{ otherwise.}
					\end{cases}
				\end{align*}
		\end{theorem}

		The following corollaries will be useful later and are direct consequences of Theorems \ref{theorem:internalOrderIsAPoset} and \ref{theorem:internalOrderFacts}.
		The first verifies our previous claim that internally active elements of a basis $B$ are always elements of the initial basis $B_{0}$, while the second computes the height of the internal order in terms of a matroid's rank and its number of coloops.
		\begin{corollary}
		\label{corollary:activeInB0}
			If $B$ is a basis of an ordered matroid $\matroid$ and if $e \in B$ is internally active, then $e$ is in the initial basis $B_{0}(\matroid)$ of $\matroid$, that is, $\IA(B) \subseteq B_{0}$ for all $B \in \bases(\matroid)$.
		\end{corollary}
		\begin{proof}
			By Theorem \ref{theorem:internalOrderFacts} the internal order of a matroid is a graded lattice and $B_0$ is the unique minimal element, so for any basis $B$ in the matroid there is a saturated chain saturated chain $B_0 \intOrder B_{1} \intOrder \cdots \intOrder B_{i} \intOrder B$ from $B_0$ to B in the internal order.
			But then we can apply part \ref{eq:ipContainment} of Theorem \ref{theorem:internalOrderIsAPoset} which assures us that for any basis $B$ and any such saturated chain in the internal order, we have~$\IA(B) \subset \IA(B_{i}) \subset \dots \subset \IA(B_{1}) \subset \IA(B_0)$.
		\end{proof}

		\begin{corollary}[\cite{lasVergnas2001active}]
		\label{corollary:htOfIntOrder}
			Let $\matroid = (E,\bases,<)$ be a rank-$r$ ordered matroid with $c$ coloops.
			Then the height of $\poset_{\Int}(\matroid)$ is equal to $r-c$.
		\end{corollary}
		\begin{proof}
			By Theorem \ref{theorem:internalOrderFacts}, the height of $\poset_{\Int}(\matroid)$ is the maximal value of $|\IP(B)|$ as~$B$ varies over $\bases(\matroid)$.
			If $e$ is a coloop of $\matroid$, then $e$ is in every basis of $\matroid$ and hence is internally active in every basis.
			Thus, the height of $\poset_{\Int}(\matroid)$ is no greater than $r-c$.
			Now let $B = S^{T}_{A}$ be a basis of $\matroid$ that is maximal in $\poset_{\Int}(\matroid)$ and suppose that there is an $e \in A$ that is not a coloop.
			Then $e = \min C^{*}(B;e)$ and there is an $f \in C^{*}(B;e)$ such that $f>e$.
			So the set $B' = B - e \cup f$ is a basis of~$\matroid$ and~$\IP(B)\subsetneq \IP(B')$.
			Thus, by Theorem \ref{theorem:internalOrderIsAPoset}, $B$ is strictly smaller than~$B'$ in~$\poset_{\Int}(\matroid)$, contradicting the fact that $B$ is maximal.
			So every basis that is maximal with respect to the internal order is of the form $B=S^{T}_{A}$ where $A$ is the set of coloops of $\matroid$, and the result follows.
		\end{proof}


\section{Internally Perfect, Abundant, and Deficient Bases}
\label{section:perfectAbundantDeficient}
	Throughout this section fix a loopless ordered matroid $\matroid = (E, \bases, \phi)$ of rank~$r$.
	Our goal is to define perfect, abundant, and deficient bases of $\matroid$.
	To this end, we begin by proving some preliminary results about minimal and principal bases.

	Recall that the minimal basis $\minBasis(I)$ of an independent set~$I$ of $\matroid$ is the lexicographically-least basis in $\bases$ containing~$I$.
	Equivalently, $\minBasis(I) = I \cup J$ where $J \subseteq B_{0}$ is the set of all minimal elements of cocircuits of $\matroid$ contained in~$E-I$.
	Thus every element of~$\minBasis(I)$ that is not in $I$ is internally active with respect to $\minBasis(I)$.

	Given a basis~$B = S^{T}_{A}$ of $\matroid$, the previous observation implies that if $I \subseteq S \cup T$ then~$\minBasis(I) \intOrder B$ with equality \Iff~$I = S \cup T$. 
	In particular, the basis $B$ is determined by its internally passive elements via $B = \minBasis(\IP(B))$.
	Equivalently,~$\minBasis(I)$ is the meet in the internal order of all bases containing~$I$.
	If $f$ is perpetually passive in $B$ (so that $f \in S$), then the basis~$B' = \minBasis(f \cup T)$ is a principal basis with $S(B') = \{f\}$.
	Moreover, as $B' \intOrder B$, we have~$T(B') \subseteq T$.	
	We record these facts in the following proposition.
	\begin{proposition}
		\label{proposition:minBases}
		Let $\matroid$ be an ordered matroid and let $B = S^{T}_{A}$ be a basis of $\matroid$.
		Then $\minBasis(S \cup T) = B$ and for any $f \in S$ the basis $B' = \minBasis(\{f\}\cup T)$ satisfies $S(B')=\{f\}$ and $T(B') \subseteq T(B)$.
	\end{proposition}

	We now prove two results concerning principal bases. 
	For the first, recall that a (finite) \defn{saturated chain} in a poset $\poset$ is a subposet whose ground set consists of elements $p_{1}, p_{2}, \dots, p_{k}$ such that $p_{i+1}$ covers $p_{i}$ for all $i \in [k-1]$.
	\begin{proposition}
		\label{proposition:principalBasesGiveChains}
		For any $f \in E - B_{0}(\matroid)$,  the set of $f$-principal bases forms a saturated chain.
	\end{proposition}
	\begin{proof}
		Write $C(B_{0};f)-f = (e_{1}, e_{2}, \cdots, e_{k})$ where $e_{i} < e_{i+1}$ with respect to the order on the ground set of $\matroid$.
		Then every $f$-principal basis of $\matroid$ is of the form~$B_{i} := B_{0} - e_{i} \cup f$ for some $i \in [k]$.
		Moreover, an element $g \in B_{i}$ is internally passive in $B_{i}$ \Iff $g>e_{i}$.
		It follows that $T(B_{i}) = T(B_{i+1}) \cup e_{i+1}$ for all~$1\le i \le k-1$. 
		Thus~{$\IP(B_{i+1})$ = $\IP(B_{i}) \cup e_{i+1}$}.
		This in turn implies that $B_{i+1}$ covers $B_{i}$ in~$\poset_{\Int}(\matroid)$ by Theorem \ref{theorem:internalOrderFacts}.
		Hence the set of $f$-principal bases forms a saturated chain.
	\end{proof}

	For a basis $B = S^{T}_{A}$ and an $f \in S$ we define the $\defn{$f$-part of T}$ to be the set~$T(B;f)$ consisting of provisionally passive elements of $B$ that are also provisionally passive in $\minBasis(f \cup T)$, that is, $T(B;f) := T(\minBasis(f \cup T))$.
	For a given basis $B$, one can read off the~$f$-part of $T$ easily from the internal order using the following proposition.
	\begin{proposition}
		\label{proposition:fPartIsMax}
		Let $B = S^{T}_{A}$ be a basis of an ordered matroid and let $f \in S$.
		Then $\minBasis(f \cup T)$ is the maximal $f$-principal basis that is less than (or equal to) $B$ in the internal order.
	\end{proposition}
	\begin{proof}
		Let $B_{1} = \minBasis(f \cup T)$ and suppose that $B_{2}$ is an $f$-principal basis such that $B_{1} \prec_{\Int} B_{2} \intOrder B$.
		Then, as in the proof of the previous proposition, there exist $e_{1}, e_{2} \in C(B_{0};f)$ such that 
		 $B_{i} = B_{0} - e_{i} \cup f$ and $e_{1}< e_{2}$.
		Since $e_{1}$ is not in~$B_{1}$ and $T$ is a subset of $B'$, $e_{1}$ is not an element of $T$.
		On the other hand, by the argument given in the previous proposition, $e_{1}$ is a provisionally passive element of $B_{2}$.
		So $e_{1} \in T(B_{2})$ and hence $\IP(B_{2}) \nsubseteq \IP(B)$.
		But then $B_{2}$ is not less than $B$ in the internal order by Theorem \ref{theorem:internalOrderIsAPoset}.
		This contradiction implies the maximality of $B_{1}$.
	\end{proof}

Given a basis $B-S^{T}_{A}$, it is natural to ask to what extent the union over $f \in S$ of the $f$-parts of $T$ cover $T$.
There are three possible answers to this question which lead us to the central definitions of this paper.
\begin{definition}
	Let $B=S^T_A$ be a basis of $\matroid$ and let $\widetilde{T}$ be the union of the $f$-parts of $T$ as $f$ runs over $S$. 
	Then $B$ is
		\begin{enumerate}
			\item (\defn{internally}) \defn{deficient} if $\widetilde{T}$ is a proper subset of $T$;
			\item (\defn{internally}) \defn{abundant} if $\widetilde{T} = T$ but $\widetilde{T}$ is not a disjoint union; and
			\item (\defn{internally}) \defn{perfect} if $\widetilde{T} = T$ and for all $f,g \in S$ with $f \neq g$ the set~$T(B;f)\cap T(B;g)$ is empty.
		\end{enumerate}
\end{definition}
	We write $\cD,\cA,$ and $ \cP$ for the set of deficient, abundant, and perfect bases of $\matroid$, respectively.
	Clearly, these sets partition the bases of $\matroid$, that is,~$\bases = \cD \smallDisjointUnion \cA \smallDisjointUnion \cP$. 
	Moreover, the set of perfect bases $\cP$ is never empty, as the next proposition shows.
	\begin{proposition}
		\label{proposition:perfectBasesExist}
		Let $\matroid$ be an ordered matroid. If $B = S^{T}_{A}$ is a basis of $\matroid$ with either $T = \emptyset$ or $|S| = 1$, then $B$ is perfect.
		If, in addition, $\matroid$ is a rank-$r$ matroid with $c$ coloops and $r-c=2$, then every basis of $\matroid$ is perfect.
	\end{proposition}
	\begin{proof}
		Let $B=S^{T}_{A}$ be a basis.
		If $T = \emptyset$, then $\widetilde{T} = \emptyset$ and so $T$ is trivially perfect.
		If $|S|=1$, then $B$ is a $f$-principal for some $f \in E - B_{0}$.
		In this case $B$ is perfect since $B = \minBasis(f \cup T)$ by Proposition \ref{proposition:fPartIsMax}.

		Now suppose $\matroid$ is a rank-$r$ ordered matroid with $c$ coloops such that $r-c=2$.
		Then, by Corollary \ref{corollary:htOfIntOrder}, the height of $\poset_{\Int}(\matroid)$ is two.
		Thus, if $B=S^{T}_{A}$ is a basis of $\matroid$, then $|S| \in \{0,1,2\}$.
		In any case, the previous paragraph implies that $B$ is perfect since either
		$|S|=1$, or~$|S|\in\{0,2\}$ in which case $T = \emptyset$.		
	\end{proof}

	In particular, the previous proposition tells us that if $\matroid$ is an unordered matroid of rank $2$, then $\matroid$ is internally perfect for any linear ordering of its ground set.
	This is far from the case for matroids of higher rank.
	In Example \ref{example:typesOfBases} we see that the matroid $\matroid$ of Example \ref{example:perfectNotPerfect} is not perfect with respect to the natural ordering of its ground set but is perfect with respect to the ordering $(2,3,1,4,5,6)$.
	Furthermore, in Example \ref{example:notPerfectAnyOrder} we supply a matroid that is not perfect with respect to any linear ordering of its ground set.

	\begin{example}
	\label{example:typesOfBases}
		Consider the two internal orders in Figure \ref{figure:poset1}.
		For the poset on the left one can check directly from the definitions that every basis is perfect except for~$B = 45^{1}$
		which has
		\[T(B;4) = T(B;5) = \emptyset \neq T = \{1\}.\]
		Thus $B$ is deficient with respect to the natural order on $[6]$.
		On the other hand, one can verify that every basis of the ordered matroid whose internal order is given by the poset on the right is perfect. \qed

	\begin{example}
    \label{example:notPerfectAnyOrder}
    	In this example we provide a matroid that is not internally perfect for any linear ordering of the ground set.
		Let $\matroid$ be the vector matroid of the columns of the matrix $M$ (over $\bbQ$) given by
		\[M :=
      \begin{pmatrix*}[r]
        1&  0&  0&  0&  0&  {-2}&  {-1}&  {1}\\
        0&  1&  0&  0&  0&  1&  1&  {1}\\
        0&  0&  1&  0&  0&  {-1}&  0&  {1}\\
        0&  0&  0&  1&  0&  {-2}&  0&  {1}\\
        0&  0&  0&  0&  1&  0&  0&  {1}\\
      \end{pmatrix*}.
		\]
		Then $\matroid$ is a rank-5 matroid with~42 bases and $h$-vector $(1,3,6,9,12,11)$.		
    	The internal order of $\matroid(M)$ (with respect to the natural order on the ground set) is given in Figure \ref{figure:r5n8} where the $f$-principal chains are colored green and the perfect (respectively abundant, deficient) bases are black (respectively blue, red).
    	For example, the  basis $B = 57^{24}_{0}$ is deficient because while $T=\{24\}$, the $f$-parts of $T$ are $T(B;5) = \emptyset$ and $T(B;7) = \{4\}$ and hence their union does not cover $T$.
    	On the other hand, the basis $B' = 57^{34}_{0}$ is abundant because the $f$-part of $T = \{34\}$ are $T(B';5) = \{3\}$ and $T(B';7) = \{3,4\}$, and their union covers $T$ but is not a disjoint union.

    	\begin{figure}[htbp]
        \centering
        \tikzstyle{every node}=[draw=black, thick, fill = white, rectangle,inner sep=2pt]
        \begin{tikzpicture}[scale=1, vertices/.style={draw, fill=black, circle, inner sep=0pt}]
               \node (0) at (-0+0,0)            {\scriptsize$\emptyset^{}_{01234}$};  
               \node (5) at (-5/4+0,8/5)        {\scriptsize${5}^{}_{0124}$};         
               \node (6) at (-5/4+5/4,8/5)      {\scriptsize${6}^{}_{0234}$};         
               \node (7) at (-5/4+5/2,8/5)      {\scriptsize${7}^{}_{0123}$};         
               \node (35) at (-25/8+0,16/5)     {\scriptsize${5}^{3}_{014}$};         
               \node (56) at (-25/8+5/4,16/5)   {\scriptsize${56}^{}_{024}$};         
               \node (57) at (-25/8+5/2,16/5)   {\scriptsize${57}^{}_{012}$};         
               \node (16) at (-25/8+15/4,16/5)  {\scriptsize${6}^{1}_{234}$};         
               \node (67) at (-25/8+5,16/5)     {\scriptsize${67}^{}_{023}$};         
               \node (47) at (-25/8+25/4,16/5)  {\scriptsize${7}^{4}_{012}$};         
               \node (235) at (-5+0,24/5)       {\scriptsize${5}^{23}_{04}$};         
               \node (156) at (-5+5/4,24/5)     {\scriptsize${56}^{1}_{24}$};         
               \node (356) at (-5+5/2,24/5)     {\scriptsize${56}^{3}_{04}$};         
               \node (567) at (-5+15/4,24/5)    {\scriptsize${567}^{}_{02}$};         
               \node (357) at (-5+5,24/5)       {\scriptsize${57}^{3}_{01}$};         
               \node (457) at (-5+25/4,24/5)    {\scriptsize${57}^{4}_{01}$};         
               \node (167) at (-5+15/2,24/5)    {\scriptsize${67}^{1}_{23}$};         
               \node (467) at (-5+35/4,24/5)    {\scriptsize${67}^{4}_{02}$};         
               \node (347) at (-5+10,24/5)      {\scriptsize${7}^{34}_{01}$};                
               \node (1235) at (-55/8+0,32/5)   {\scriptsize${5}^{123}_{4}$};         
               \node (1356) at (-55/8+5/4,32/5) {\scriptsize${56}^{13}_{4}$};         
               \node (2356) at (-55/8+5/2,32/5) {\scriptsize${56}^{23}_{4}$};         
               \node (1567) at (-55/8+15/4,32/5){\scriptsize${567}^{1}_{2}$};         
               \node (3567) at (-55/8+5,32/5)   {\scriptsize${567}^{3}_{0}$};         
               \node (4567) at (-55/8+25/4,32/5){\scriptsize${567}^{4}_{0}$};         
               \node (2357) at (-55/8+15/2,32/5){\scriptsize${57}^{23}_{0}$};         
               \node (2457) at (-55/8+35/4,32/5){\scriptsize$\color{red}{57}^{24}_{0}$};          
               \node (3457) at (-55/8+10,32/5)  {\scriptsize$\color{blue}{57}^{34}_{0}$};        
               \node (1467) at (-55/8+45/4,32/5){\scriptsize${67}^{14}_{2}$};         
               \node (3467) at (-55/8+25/2,32/5){\scriptsize${67}^{34}_{0}$};         
               \node (2347) at (-55/8+55/4,32/5){\scriptsize${7}^{234}_{0}$};         
               \node (13567) at (-25/4+0,8)     {\scriptsize${567}^{13}_{}$};         
               \node (14567) at (-25/4+5/4,8)   {\scriptsize${567}^{14}_{}$};         
               \node (23567) at (-25/4+5/2,8)   {\scriptsize${567}^{23}_{}$};         
               \node (24567) at (-25/4+15/4,8)  {\scriptsize$\color{red}{567}^{24}_{}$};          
               \node (34567) at (-25/4+5,8)     {\scriptsize$\color{blue}{567}^{34}_{}$};        
               \node (12357) at (-25/4+25/4,8)  {\scriptsize${57}^{123}_{}$};         
               \node (13457) at (-25/4+15/2,8)  {\scriptsize$\color{red}{57}^{134}_{}$};          
               \node (23457) at (-25/4+35/4,8)  {\scriptsize$\color{blue}{57}^{234}_{}$};        
               \node (13467) at (-25/4+10,8)    {\scriptsize${67}^{134}_{}$};         
               \node (23467) at (-25/4+45/4,8)  {\scriptsize${67}^{234}_{}$};         
               \node (12347) at (-25/4+25/2,8)  {\scriptsize${7}^{1234}_{}$};         
       \foreach \to/\from in {0/5, 0/6, 0/7, 5/35, 5/56, 5/57, 6/56, 6/16, 6/67, 7/57, 7/67, 7/47,35/235, 35/356, 35/357, 56/156, 56/356, 56/567, 57/567, 57/357, 57/457, 16/156, 16/167, 67/567, 67/167, 67/467, 47/457, 47/467, 47/347,235/1235, 235/2356, 235/2357, 156/1356, 156/1567, 356/1356, 356/2356, 356/3567, 567/1567, 567/3567, 567/4567, 357/3567, 357/2357, 357/3457, 457/4567, 457/2457, 457/3457, 167/1567, 167/1467, 467/4567, 467/1467, 467/3467, 347/3457, 347/3467, 347/2347,1235/12357, 1356/13567, 2356/23567, 1567/13567, 1567/14567, 3567/13567, 3567/23567, 3567/34567, 4567/14567, 4567/24567, 4567/34567, 2357/23567, 2357/12357, 2357/23457, 2457/24567, 2457/23457, 3457/34567, 3457/13457, 3457/23457, 1467/14567, 1467/13467, 3467/34567, 3467/13467, 3467/23467, 2347/23457, 2347/23467, 2347/12347}
       \draw [-] (\to)--(\from);
       \foreach \to/\from in {5/35, 6/16, 7/47, 35/235, 47/347, 235/1235, 347/2347, 2347/12347}
       \draw [-, green] (\to)--(\from);
       \end{tikzpicture}
    
    	\caption{{Perfect}, {\color{blue} abundant}, and {\color{red}deficient} bases of a rank-$5$ matroid on $8$ elements}
        \label{figure:r5n8}
    	\end{figure}
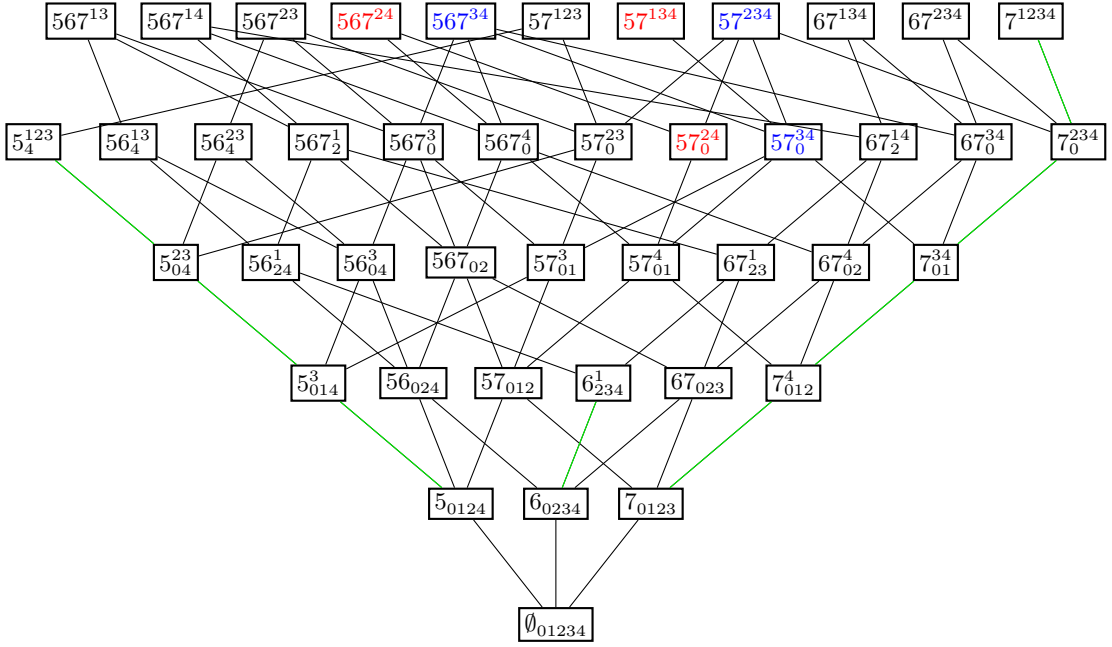
    \end{example}
    	Using, for example, the \texttt{Macaulay2} package \texttt{Posets} (see \cite{M2}), one can perform a brute force computation to show that none of the $8!$ linear orders on the ground set of $\matroid$ yields a perfect ordered matroid.
    
    	We now use this example to make some observations in order to motivate upcoming results. 
    	First notice that~if $B=S^{T}_{A}$ is a perfect basis of $\matroid$, then $B$ covers exactly $|S|$ bases in $\poset_{\Int}$ and that $B$ can be expressed as the join of principal bases in a unique way.
    	For example, the perfect basis $567^{13}$ covers the bases $56^{13}_{4}, 567^{1}_{2},$ and~$567^{3}_{0}$ and can only be written as the join of the principal bases as follows: 
    		\begin{align*}
    			B &= \bigvee_{f\in S} \minBasis(f \cup T)
    			\\&= 5^{3}_{014} \vee 6^{1}_{234} \vee 7_{0123}.
    		\end{align*}
    	Also notice that if $B$ is a perfect basis and $B' \intOrder B$, then $B'$ is perfect.

    	When $B$ is abundant, it covers more than $|S|$ bases and can be expressed as the join of $f$-principal bases in a number of ways.
    	For example, the basis $57^{34}_{0}$ covers~$3$ bases and can be written as such a join in three distinct ways:
    		\[B \hspace{11pt}=\hspace{11pt} 5^{3}_{014} \vee 7^{34}_{01}
    			\hspace{11pt}=\hspace{11pt}5^{3}_{014} \vee 7^{4}_{012}
    			\hspace{11pt}=\hspace{11pt}5_{0124} \vee 7^{34}_{01}.\]
    	
    	Finally, notice that any deficient basis of $\matroid$ cannot be expressed as the join of~$f$-principal bases. 
    	\qed
	\end{example}

	We now proceed to prove that the properties verified in the previous example hold in general.
	We begin by characterizing the three types of bases in terms of the interplay between principal bases and the join operator in the internal order.
	For a basis $B = S^{T}_{A}$ of an ordered matroid $\matroid$ let
			\begin{equation}
				\label{equation:decomposeIntoPrincipal}
				B' := \bigvee_{f \in S} \minBasis(f \cup T).
			\end{equation}
	Let us call $B$ \defn{decomposable} if $B = B'$, and \defn{undecomposable} otherwise.
	Moreover, call $B$ \defn{uniquely decomposable} if $B$ is decomposable and \eqref{equation:decomposeIntoPrincipal} is the unique way to write $B$ as the join of $f$-principal bases for $f\in S$, and \defn{multi-decomposable} otherwise. 
	Since every basis of $\matroid$ is of exactly one of these three types, we obtain a partition of $\bases(\matroid)$ which we call the \defn{decomposability partition}.
	With this terminology in hand we can now state our characterizations of perfect, abundant, and deficient bases.
	\begin{proposition}
	\label{proposition:joinsOfProjections}
		Let $B$ be a basis of an ordered matroid $\matroid$.
		Then $B$ is perfect (abundant, deficient) if and only if $B$ is uniquely (respectively, multi-, un-) decomposable.		
	\end{proposition}
	\begin{proof}
		Note that we have two partitions of the bases of $\matroid$.
		On the one hand we have $\bases = \cD \smallDisjointUnion \cA \smallDisjointUnion \cP$, and on the other we have the decomposability partition.
		Therefore it is sufficient to prove the ``only if'' direction of the proposition.
			
		First we deal with the trivial case $B=B_{0}(\matroid)$.
		In this case $B$ is perfect by Proposition \ref{proposition:perfectBasesExist}.
		Since $B$ is the least element of the lattice $\poset_{\Int}(\matroid)$, the typical convention for empty joins in lattices yields $B = \bigvee \emptyset = B'$. 
		As this expression is clearly unique, $B_{0}$ is uniquely decomposable.

		Now we turn to the general case.
		Let $B \neq B_{0}$ be a basis of $\matroid$. 
		By repeated application of the expression for the join of two bases in Theorem \ref{theorem:internalOrderFacts} we may write~$B'$ as 
			\[
				B' = \minBasis\left(\bigcup_{f \in S}\intPassive \big(\minBasis(f \cup T\big)\right).
			\]
		For every $f \in S$ the internally passive elements of $\minBasis(f \cup T) $ are internally passive in $B$ by Proposition \ref{proposition:minBases}.
		It follows that $\bigcup_{S}\intPassive (\minBasis(f \cup T)) \subseteq \intPassive(B)$ and, in particular, that the union is an independent set in $\matroid$.
		Moreover, Proposition~\ref{proposition:minBases} implies that the internally passive elements of $\minBasis(f \cup T)$ are the elements of~$f \cup T(B;f)$ and so we have
			\[B' = \minBasis\left(S \cup \bigcup_{f \in S} T(B;f)\right).\]

		If the basis $B$ is deficient, then $\intPassive(B') = S \cup \widetilde{T} \subsetneq \intPassive(B)$ which proves that~$B'$ is strictly smaller than $B$ in the internal order.
		Now suppose there is a collection of principal bases whose join is $B$.
		Then this collection must be of the form 
			\[\{B_ f \suchthat f \in S \text{ and } B_ f \text{ is $f$-principal}\}\]
		and there must be at least one $f \in S$ for which $\minBasis(f \cup T) \prec_{\Int} B_ f$.
		But then~$\minBasis(f \cup T) \prec_{\Int} B_ f \intOrder B$ which contradicts the fact that~$\minBasis(f \cup T)$ is the maximal $f$-principal smaller than $B$ in the internal order.
		Hence $B$ cannot be written as the join of principal bases.

		Now fix a basis $B$ that is not deficient.
		Then $\intPassive(B') = \intPassive(B)$ and so the bases~$B$ and $B'$ coincide.
		Note that in this case we can write $B = \bigvee_{f \in S} P_ f$ for any collection of principal bases $\{P_ f \suchthat S(P_ f) = f\}$ such that $\bigcup T(P_ f) = T$.
		If $B$ is a perfect basis, we must have $P_ f = \minBasis(f \cup T)$ and so \eqref{equation:decomposeIntoPrincipal} is a unique expression for~$B$ as the join of principal bases.
		On the other hand, if $B$ is abundant then we can write~$B = \bigvee P_ f$ where $S(P_ f)=f$ and 
		\[T(P_ f) = \{t \in T(B;f) \suchthat f = \min\{g \in S \suchthat t \in T(B;g)\}\}.\] 
		As the bases $P_ f$ are all principal and the sets $T(P_ f)$ partition $T$ it follows that~$\bigvee P_ f$ is an expression of $B$ as the join of principal bases and that this expression is different from \eqref{equation:decomposeIntoPrincipal}, completing the proof.
	\end{proof}

	Proposition \ref{proposition:joinsOfProjections} gives us one way to use the internal order of $\matroid$ to determine if a basis $B$ is perfect.
	The next proposition gives us another.

	\begin{proposition}
		\label{proposition:perfectionFiltersDown}
		Every basis in the downset of an internally perfect basis in~$\poset_{\Int}(\matroid)$ is internally perfect.
	\end{proposition}
	\begin{proof}
		It is sufficient to prove that every basis covered by an internally perfect basis is internally perfect.
		For this let $B = S^{T}_{A}$ be a perfect basis of an ordered matroid~$\matroid$ and let $B' = B - e \cup a$ be covered by $B$.
		As $\poset_{\Int}$ is a graded poset we must have $|\intPassive(B')| = |\intPassive(B)|-1$ and so either 
			\begin{inparaenum}
				\item $e \in S$ and $T(B')=T(B)$, 
				or~\item $e \in T$ and $S(B')=S(B)$.
			\end{inparaenum}
		In the first case we have $T(B;e) = \emptyset$ since $B' \isCovered B$.
		It then follows that
			\[T(B') = T(B) = \disjointUnion_{f \in S(B)}T(B;f) = \disjointUnion_{f \in S(B')}T(B;f),\]
		as desired.
		In the second case, since~$B$ is perfect and $e \in T$ we have a unique $f \in S$ such that $e \in T(B;f)$.
		But then 
		\begin{align*}
			T(B') &= T(B) -e 
			\\ &= (T(B;f) - e) \cup \disjointUnion_{g \in S-\{f\}}T(B;g) .
		\end{align*}
		This union is disjoint, so it follows that $B'$ is perfect in this case.
	\end{proof}

	In the subsequent sections we will be interested in perfect ordered matroids, that is, ordered matroids whose bases are all perfect.
	Note that Proposition \ref{proposition:perfectionFiltersDown} supplies us with a useful computational tool to verify that a given ordered matroid is perfect insofar as it implies that to check that every basis of an ordered matroid is perfect, it is enough to check that all coatoms of the internal order (i.e., all bases covered by the artificial top element $\hat{1}$) are perfect.
	Also note that implicit in the proof of Proposition \ref{proposition:perfectionFiltersDown} is the fact that a perfect basis $B=S^{T}_{A}$ covers at most $|S|$ many bases in the internal order.
	In \cite{lasVergnas2001active}, Las Vergnas shows that the number of bases covered by an arbitrary basis $B$ in the internal order is no greater than the number of internally passive elements of $B$.
	In the next proposition we compute the exact number of bases covered by a perfect basis $B$ in the internal order.

	\begin{proposition}
		\label{proposition:characterizationsByCovers}
		If $B = S^{T}_{A}$ is a perfect basis of an ordered matroid $\matroid$, then it covers exactly $|S|$ bases in $\poset_{\Int}(\matroid)$.
		
	\end{proposition}
	\begin{proof}
		Let $B = S^{T}_{A}$ be a perfect basis of $\matroid$.
		Then Proposition \ref{proposition:joinsOfProjections} yields the following unique expression of $B$ as the join of principal bases:
		\begin{equation}
		\label{equation:uniqueDecomp}	
		B = \bigvee_{f \in S} \minBasis(\{f\} \cup T).
		\end{equation}		
		For $g\in S$ write $A_{g} := \bigvee_{f \in S - \{g\}} \minBasis(\{f\} \cup T)$ and consider the interval~$[A_{g},B]$ in $\poset_{\Int}$.
		We claim that the interval $[A_{g},B]$ is isomorphic to the chain~$[B_{0}, B']$ where~$B' = \minBasis\big(\{g\} \cup T(B;g)\big)$.
		To see this note that a basis $B''$ in the half-open interval $(A_{g},B]$ satisifies 
			\begin{align*}
				S(A_{g}) &\subsetneq S(B'') \subseteq S(B)
				\\T(A_{g}) &\subseteq T(B'') \subseteq T(B).
			\end{align*}
		As Equation \eqref{equation:uniqueDecomp} is the unique expression of $B$ as the join of principal bases, it follows that $S(B'')- S(A_{g}) = \{g\}$ and $T(B'')-T(A_{g}) \subseteq T(B;g)$.
		It follows that~$B''$ is the join of $A_{g}$ and some $g$-principal basis in $[B_{0},B']$, proving the claim.
		
		Given distinct $f,g \in S$, the bases $A_ f$ and $A_{g}$ are incomparable in the internal order of $\matroid$ and their join is $B$. 
		If we write $B_{g}$ for the basis in $[A_{g},B]$ that is covered by $B$, then
		it follows that $B_ f \neq B_{g}$ for distinct $f,g \in S$.
		This implies that $B$ covers at least $|S|$ bases in $\poset_{\Int}$.
		Since the proof of the previous proposition implies that $B$ covers at most $|S|$ bases, the result follows.
	\end{proof}

\section{Minors of Perfect Matroids}
\label{section:minors}

In the previous section we studied local properties of internally perfect bases of an ordered matroid.
In this and subsequent sections, we take a global perspective and investigate \defn{internally perfect ordered matroids}, that is, ordered matroids whose bases are all internally perfect.
For brevity's sake we typically call such an ordered matroid a perfect matroid.
Moreover, we call an unordered matroid (\defn{internally}) \defn{perfect} if there is some linear ordering of its ground set so that the resulting ordered matroid is internally perfect.
Our goal in this section is to study the minors of perfect matroids in order to prove the following theorem.
\begin{theorem}
\label{theorem:minorClosedIfDelOffB0}
	Let $\matroid = (E, \bases, \phi)$ be an internally perfect ordered matroid with initial basis $B_{0}$ and let $F_{1}$ and $F_{2}$ be disjoint subsets of $E$ such that any element of~$F_{2} \cap B_{0}$ is a coloop.
	Then the minor $\matroid /F_{1}\setminus F_{2}$ is internally perfect with respect to the ordering of its ground set induced by the order of $\matroid$.
\end{theorem}

For the remainder of this section let us fix an ordered internally perfect matroid~$\matroid = (E, \bases, \phi)$ and an element $e \in E$.
In the trivial case where $\matroid$ has exactly one basis, Proposition \ref{proposition:perfectBasesExist} implies that $\matroid$ is internally perfect.
Moreover, the same proposition implies that every minor of $\matroid$ is internally perfect.
So without loss of generality we assume hereafter that $|\bases(\matroid)|>1$.

We prove Theorem \ref{theorem:minorClosedIfDelOffB0} in three main steps in order of increasing complexity.
First we show that deleting or contracting a loop or coloop of $\matroid$ preserves perfection.
Then we prove that $\matroid-e$ is perfect whenever $e \notin B_{0}(\matroid)$; see Section \ref{PMdeletions}.
Finally, in Section \ref{PMcontractions} we demonstrate that~$\matroid/e$ is  perfect whenever~$\matroid$ is.

We begin with the case when $e \in E$ is either a loop or a coloop.
It is a straightforward exercise to prove that in this case the posets $\poset_{\Int}(\matroid),\poset_{\Int}(\matroid/e),$ and $\poset_{\Int}(\matroid-e)$ are all isomorphic.
This fact and some simple comparisons of~$\STA$-decompositions of bases in the various matroids are enough to show that if $\matroid$ is internally perfect, then so are $\matroid-e$ and $\matroid/e$.

\subsection{Perfect Matroids and Deletion}
\label{PMdeletions}

Having fixed an internally perfect ordered matroid $\matroid$ and an element $e \in E(\matroid)$ as in the introduction to this section, let $\cN := \matroid - e$.
Our goal is to prove that $\cN$ is perfect whenever $e$ is not in the initial basis of $\matroid$.
We will need the following well-known (and easy to prove) facts characterizing the bases, circuits and cocircuits of $\cN$ in terms of those of $\matroid$.

	\begin{proposition}
	\label{proposition:delBasicFacts}
		Let $\matroid = (E, \bases)$ be a matroid, $e \in E$, and $\cN = \matroid - \{e\}$.
		Then 
			\begin{enumerate}[label=\upshape(\roman*)]
			\item \label{eq:delBases} $\bases(\matroid - e) = \{B \in \bases(\matroid) \suchthat e \notin B\}$,
			\item \label{eq:delCircuits} $\circuits(\cN) = \{C \subseteq E-e \suchthat C \in \circuits(\matroid)\}$, and
			\item \label{eq:delCocircuits} $\cocircuits(\cN)$ is the set of minimal nonempty members of the set 
				\[\{C^{*}-\{e\} \suchthat C^{*} \in \cocircuits(\matroid)\}.\]
			\end{enumerate}
	\end{proposition} 

As any set $B \subseteq E-e$ that is a basis of $\cN$ is also a basis of $\matroid$, it is helpful to introduce notation to clarify the matroid of which we are considering $B$ a basis.
For a basis $B = S^{T}_{A} \in \bases(\cN)$ we write $B' = S'^{T'}_{A'}$ for the corresponding basis of~$\matroid$.
Define the set $X(B')$ to be the set of all internally passive elements $b$ of $B'$ such that $\{e' \in C^{*}(B';b) \suchthat e' < b\} = \{e\}$.
The next lemma computes the internally active (respectively, passive) elements of a basis $B \in \bases(\cN)$ in terms of the active (respectively, passive) elements of $B'$.


	\begin{lemma}
	\label{lemma:IA/IPBofDeletion}
		Let $B \in \bases(\cN)$ and $B'$ be the corresponding basis in $\bases(\matroid)$.
		Then we have $\IP(B) = \IP(B')-X(B')$ and $\IA(B) = \IA(B')\cup X(B')$.	
		Moreover, whenever~$e \notin B_{0}(\matroid)$, then $X = \emptyset$, i.e., the $\text{STA}$-decompositions of $B$ and $B'$ coincide.	
	\end{lemma}
	\begin{proof}
		An element $b$ is internally passive in $B'$ if and only if $b \neq \min C^{*}(B';f)$.
		Since $C^{*}(B;f) = C^{*}(B';f)-e$, it follows that $b \in \IP(B)$ \Iff $b \in \IP(B')$ and $e$ is not the only element of the set $\{e' \in C^{*}(B';b) \suchthat e' < b\}$, proving the first claim.
		The second claim follows immediately since $\IA(B) = B - \IP(B)$.

		Now suppose $e \notin B_{0}(\matroid)$.
		Then $e$ is not internally active in any basis in which it occurs.
		We claim that this implies $X(B') = \emptyset$.
		Otherwise there is a $b \in B'$ such that $\{e' \in C^{*}(B';b) \suchthat e' < b\} = \{e\}$, from which it follows that $e$ is internally active in the basis $B' - b \cup e$.
		This contradiction shows that $X(B')$ must be empty and hence that $\IP(B) = \IP(B')$ and $\IA(B) = \IA(B')$.
		The fact that the~$\STA$-decompositions of $B$ and $B'$ coincide now follows from the observation that $B_{0}(\cN) = B_{0}(\matroid)$ whenever $e \notin B_{0}(\matroid)$.
	\end{proof}	


With the previous lemma in hand it is easy to see that $\cN$ is perfect whenever~$\matroid$ is perfect and $e \notin B_{0}(\matroid)$.
Let $B=S^{T}_{A} \in \bases(\cN)$ and $B'=S'^{T'}_{A'}$ for the corresponding basis in $\bases(\matroid)$. 
By Lemma \ref{lemma:IA/IPBofDeletion}, the bases $B$ and $B'$ have identical~$\STA$-decompositions and so
			\[T = T' = \disjointUnion_{f\in S'} T(B';f) = \disjointUnion_{f\in S} T(B;f)\]
where the last equation follows from the fact that 
	\[T(B';f) = T(\minBasis_{\matroid}(f \cup T')) = T(\minBasis_{\cN}(f \cup T)) = T(B;f).\]			
Thus we have proven the following corollary.

	\begin{corollary}
	\label{corollary:delIsPerfect}
		Let $\matroid$ be an internally perfect ordered matroid.
		Then $\cN = \matroid - e$ is an internally perfect ordered matroid with respect to the ordering on the ground set of $\matroid$ whenever $e \in E - B_{0}(\matroid)$.
	\end{corollary}

Taken together, Corollary \ref{corollary:delIsPerfect} and the remarks at the end of the introduction to this section show that $\cN = \matroid - e$ is perfect whenever $e$ is a (co)loop or is not an element of the initial basis of $\matroid$.
This result is the best possible in the sense that there are internally perfect ordered matroids for which deleting an element of their initial bases yields an ordered matroid that is not perfect with respect to the order on the original matroid, as the next example shows.

	\begin{example}
	\label{example:notMinorClosed}
		Consider the ordered vector matroid $\matroid$ on seven elements given by the matrix
			\[M:=
			\begin{pmatrix}
				1 & 0 & 0 & 0 & 0 & 0 &  1 \\
				0 & 1 & 0 & 0 & 1 & 1 &  2 \\
				0 & 0 & 1 & 0 & 0 & 1 &  0 \\
				0 & 0 & 0 & 1 & 0 & 0 & -1 		
			\end{pmatrix}.
			\]
		One can verify that each of the 3 maximal bases in $\poset_{\Int}(\matroid)$ are perfect, and hence that $\matroid$ is internally perfect by Proposition \ref{proposition:perfectionFiltersDown}.
		Note that the second column is neither a loop nor coloop of $\matroid$ and that $B = \{1,4,6,7\}$ is a basis of $\cN = \matroid- \{2\}$.
		The $\STA$-decomposition of $B$ is $67^{4}_{1}$ but $T(B;6) = T(B;7) = \emptyset$ and so $\cN$ is not internally perfect with respect to the order of its ground set induced by the order of $\matroid$.	  

		Though the ordered matroid $\cN$  is not internally perfect with respect to the linear order $E(\cN) = (1,3,4,5,6,7)$, the underlying unordered matroid is internally perfect.
		To see this note that the elements $2$ and $5$ of $\matroid$ are parallel and hence the unordered matroids $\matroid-\{2\}$ and $\matroid - \{5\}$ are isomorphic.
		Since $5$ is not in $B_{0}(\matroid)=\{1,2,3,4\}$, Corollary \ref{corollary:delIsPerfect} implies that $\matroid-\{5\}$ is internally perfect.
		Thus the ordered matroid $\cN' = \matroid - \{2\}$ with $E(\cN') = (1,5,3,4,6,7)$ is internally perfect
	\end{example}

	It is natural to ask whether the phenomenon witnessed in the previous example is borne out in the general case.
	More precisely, is there always a reordering of the ground set of the ordered matroid $\cN = \matroid - e$ (where $e \in B_{0}(\matroid)$) such that $\cN$ is internally perfect?
	Computational evidence verifies that this is true for matroids with small rank ($r(\matroid) \le 4$).
	Establishing this in the general case would prove the following conjecture. 

	\begin{conjecture}
	\label{conjecture:unorderedDelAlwaysPerfect}
		The class of unordered internally perfect matroids is closed under deletion.
	\end{conjecture}

\subsection{Contractions of Perfect Matroids}
\label{PMcontractions}

	We have seen that internal perfection is preserved when deleting elements from a perfect matroid as long as the elements deleted do not belong to the initial basis.
	Our present goal is to show that internal perfection and matroid contraction are more compatible in the sense that any contraction of an internally perfect is internally perfect.
	As before, we fix an ordered internally perfect matroid~$\matroid = (E, \bases, \phi)$ and an element $e \in E$.
	The proof presented here is essentially a detailed analysis of $\STA$-decompositions of the minimal bases $\minBasis_{\cN}(I)$ and $\minBasis_{\matroid}(I)$, where $\cN = \matroid/e$ and $I$ is an independent set of $\cN$.

	First we collect some well-known results concerning matroid contractions (of unordered matroids).
	\begin{proposition}
	\label{proposition:basicContractionFacts}
		Let $\matroid = (E, \matroid)$ be a matroid, $e \in E$, and $\cN = \matroid/e$. Then
			\begin{enumerate}[label=(\alph*), font = \upshape]
				\item \label{cont1}  a set $I \subseteq E-e$ is independent in $\cN$ \Iff $I\cup e$ is independent in~$\matroid$;
				\item \label{cont2} the circuits of $\cN$ are the minimal nonempty sets of~$\{C- e \suchthat C \in \circuits(\matroid)\}$; and
				\item \label{cont3} the set of cocircuits of $\cN$ is $\{C^{*} \subseteq E-e \suchthat C^{*}\in \cocircuits(\matroid)\}$.
			\end{enumerate}
	\end{proposition}

	As $\matroid = (E, \bases, \phi)$ is an ordered matroid, the restriction of $\phi$ to $E-e$ makes the contraction $\cN:= \matroid/e$ an ordered matroid, which in turn allows to talk about internal activity and passivity in $\cN$.
	
	Let $B$ be a basis of $\cN$.
	Recall from Section \ref{subsection:internalOrder} that an element $b \in B$ is internally active in $\matroid$ \wrt $B$ \Iff there is a cocircuit $C^{*}$ of $\matroid$ contained in $(E-\{e\})-B \cup b$ such that $b = \min C^{*}$.
	It follows from Proposition~\ref{proposition:basicContractionFacts}\ref{cont3} that such a cocircuit $C^{*}$ of $\cN$ is also a cocircuit of $\matroid$.
	Moreover, since $b = \min C^{*} \subseteq E - (B \cup e) \cup b$, we have that $b \in \IA_{\matroid}(B \cup e)$.
	Since any basis is the disjoint union of its sets of internally active and internally passive elements, we have the shown the following proposition.


	\begin{proposition}
	\label{proposition:IAandIPofContractions}
	Let $\matroid$ be an ordered matroid, $e \in E(\matroid)$, and $\cN=\matroid/e$.
	Then, for any basis $B$ of $ \cN$, we have 
	\begin{enumerate}[label=(\alph*), font = \upshape]
	 	\item $\IA_{\cN}(B) \subseteq \IA_{\matroid}(B \cup e)$ with equality \Iff $e \notin \IA(B \cup e)$, and
	 	\item $\IP_{\cN}(B) \subseteq \IP_{\matroid}(B \cup e)$ with equality \Iff $e \in \IA(B \cup e)$.
	\end{enumerate}
	\end{proposition}

	Let us illustrate Proposition \ref{proposition:IAandIPofContractions} in the case in which it will be predominantly used, namely, when $B = \minBasis_{\cN}(I)$ for some independent set $I \in \cN$.
	In this case the basis $B'' := B \cup e \in \bases(\matroid)$ is equal to $\minBasis_{\matroid}(I \cup e)$, and the proposition tells us that internally active (respectively, passive) elements of $B$ remain active (respectively, passive) in $B \cup e$.

	There is a second basis of $\matroid$ that is natural to consider when dealing with the basis $B = \minBasis_{\cN}(I)$, namely, $B' := \minBasis_{\matroid}(I)$.
	Note that if $e \in B'$, then the bases $B'$ and $B'' = \minBasis_{\matroid}(I \cup e)$ coincide.
	On the other hand, if $e \notin B'$, then evidently $B' \lessLex B''$.
	In this case, since the basis $B$ is contained in  $B'\cap B''$, we have $B'' = B' - a \cup e$ for some $a \in B' - (I \cup e)$.
	Moreover, as $B''$ is the minimum basis in $\matroid$ containing $I \cup e$, the element $a$ must equal $\max(C(B';e) - (I \cup \{e\}))$.
	It is then straightforward to show that $a$ is an internally active element of $B'$.
	The next proposition records these facts.

	\begin{proposition}
	\label{proposition:relateMinBas}
		Let $\matroid = (E,\bases, \phi)$ be an internally perfect matroid, $e \in E$, and let $\cN := \matroid/e$.
		For $I \in \independents(\matroid)$ such that $I \cup e$ is also independent, let $B'$ denote the basis $\minBasis_{\matroid}(I)$ and write $J = I - e$.
		Then
			\[
				\minBasis_{\cN}(J) = 
					\begin{cases}
						B' - e \text{ if $e \in B'$, and} 
						\\ B' - a \text{ otherwise}
					\end{cases}
			\]
		where $a := \max \left(C(B';e) - (I \cup \{e\}) \right)$ is an internally active element of $B'$.
	\end{proposition}





	The bases $B, B',$ and $ B''$ described above satisfy relations that will be crucial in proving that the matroid $\cN$ is internally perfect.
	Let us write $B=S^{T}_{A}, B'=S'^{T'}_{A'}$, and $B''=S''^{T''}_{A''}$.
	The following technical lemmas provide the relations between the triples $(S,T,A), (S',T',A')$ and $(S'',T'',A'')$.
	First, in Lemma \ref{proposition:relateSTAwe}, we consider the easier case when $e \in B'$ (so that $(S',T',A')=(S'',T'',A'')$).
	Then we deal with the case when $e \notin B'$ in Lemma \ref{proposition:relateSTAnoe}.

	\begin{lemma}
	\label{proposition:relateSTAwe}
		Let $\matroid = (E,\bases, \phi)$ be a perfect ordered matroid, let $e \in E$, and let $\cN=\matroid/e$.
		Fix $I \in \independents(\cN)$ and let $B,B'$ and $B''$ be as above.
		If $e \in B'$, then $B' = B''$ and 
			\[
				(S,T,A) = 
				\begin{cases}
					(S',T',A'-e) \text{ if $e\in A'$,}
					\\(S',T'-e,A') \text{ if $e \in T'$,}
					\\(S'-e,T',A') \text{ if $e \in S'$ and $T(B';e)=\emptyset$, and}
					\\(S'-e \cup a,T'-a,A') \text{ if $e \in S'$ and $T(B';e)\neq\emptyset$}
				\end{cases}			
			\]
			where $a := \max(C(B_{0}(\matroid); e) - \{e\})$.
	\end{lemma}
	\begin{proof}
		Adopt the notation in the statement of the proposition and let $e \in B'$.
		Evidently $B'=B''$ and so $(S',T',A')=(S'',T'',A'')$.
		So by Proposition \ref{proposition:IAandIPofContractions} we have either 
		\begin{inparaenum}[\upshape (a)]
			\item \label{eq:case1} $A = A'$ and $\IP(B) =\IP(B)-e$, or
			\item \label{eq:case2} $A = A'-e$ and~$\IP(B) =\IP(B)$.
		\end{inparaenum}

		Suppose $e \in B_{0}(\matroid)$ so that $B_{0}(\cN)=B_{0}(\matroid)-e$.
		If \eqref{eq:case1} holds, then $e \in T'$ and direct computations show that $S=S'$, $T=T'-e$, and $A=A'$ in this case.
		For example, we~have
			\begin{align*}
				S &= \IP(B)-B_{0}(\cN)
				\\&= (\IP(B')-\{e\})-(B_{0}(\matroid)-\{e\})
				\\&= (\IP(B')- B_{0}(\matroid)) - \{e\}
				\\&= S'
			\end{align*}
		where the last equation follows since $e \in B_{0}(\matroid)$ and hence $e \notin S'$.
		On the other hand, if \eqref{eq:case2} holds then $e \in A'$ and it follows immediately that $S=S', T=T'$ and~$A=A'-\{e\}$.
		This proves the proposition in case $e \in B_{0}(\matroid)$.

		Now suppose $e \notin B_{0}(\matroid)$.
		Then $e \in B'$ implies that $e \in S'$ and so case \eqref{eq:case1} above holds.
		Moreover, the inital basis of $\cN$ is $B_{0}(\cN) = B_{0}(\matroid) - a$ where $a$ is the maximal element of $C(B_{0}(\matroid);e)- \{e\}$.
		So $A = A'$ and 
			\begin{align*}
				T &= \IP(B)\cap B_{0}(\cN)
				\\&= (\IP(B') - \{e\})\cap (B_{0}(\matroid)-a)
				\\&= (\IP(B') \cap B_{0}(\matroid))- \{e,a\}
				\\&= T' - \{a\}
			\end{align*}
		where the final equation follows since $e \in S'$ implies $e \notin T'$.
		A similar computation shows that $S = (S'-e) \cap(\IP(B') \cap a)$.
		Finally, as $a \in T(B';e)$ if and only if~$T(B';e)\neq \emptyset$, we have $S=S'-e \cup a$ and $T=T'-a$ when $T(B';e)\neq \emptyset$.
		If~$T(B';e)= \emptyset$, then $S=S'-e$ and $T=T'$, as desired.
	\end{proof}


	The preceding lemma gives a case-by-case analysis of the $\STA$-decomposition of~$\minBasis_{\cN}(I)$ when $e \in \minBasis_{\matroid}(I)$.
	The next lemma provides a similar analysis when $e \notin \minBasis_{\matroid}(I)$.


	\begin{lemma}
	\label{proposition:relateSTAnoe}
		Let $\matroid = (E,\bases, \phi)$ be a perfect ordered matroid, let $e \in E$, and let~$\cN=\matroid/e$.
		Fix $I \in \independents(\cN)$ and let $B, B'$ and $B''$ be as above.
		Then if~$e$ is not in $B'$ then $e \in \IP(B'')$ and the decompositions of $B, B'$ and $B''$ are related as follows:~$(S,T,A)$ and $(S'',T'',A'')$ satisfy the relations of Lemma \ref{proposition:relateSTAwe} and
			\[
				(S,T,A) = 
					\begin{cases}
						(S'-\{e\}, T', A'-a') 
							\text{ if $e \in S''$ and $T_{e}= \emptyset$,}
						\\(S' \cup a, T' \cup T_{e}-a, A'-T_{e}-a') 
							\text{  if $e \in S''$ and $T_{e}\neq \emptyset$, and}
						\\ (S', T' \cup \{t \in T_{f} \suchthat t\le e\}, A' - \{t \in T_{f} \suchthat t < e\}) \text{ if $e \in T''$},
					\end{cases}
			\]
		where $T_{e}:= T(B'';e)$, $T_{f}:= T(B'';f)$, $a := \max(C(B_{0}(\matroid);e)-\{e\})$ and $a'$ is the maximal element of $C(B';e)-(I \cup \{e\})$.
	\end{lemma}
	\begin{proof}
		Suppose $e \notin B'$, so that $B' \neq B$.
		Lemma \ref{proposition:relateSTAwe} assures us that 
			
			\begin{equation}
			\label{eq:decomps}
				(S,T,A) = 
				\begin{cases}
					(S'',T'',A''-e) \text{ if $e\in A''$,}
					\\(S'',T''-e,A'') \text{ if $e \in T''$,}
					\\(S''-e,T'',A'') \text{ if $e \in S''$ and $T_{e}=\emptyset$, and}
					\\(S''-e \cup a,T''-a,A'') \text{ if $e \in S''$ and $T_{e}\neq\emptyset$}.
				\end{cases}			
			\end{equation}
		By the minimality of $B''$, it is straightforward to show that  $B'' = B' - \{a'\} \cup \{e\}$ and hence that $B' = B \cup a'$.
		Since $I \subseteq B' \cap B''$ and $B'= \minBasis_{\matroid}(I)$, the basis~$B'$ is lexicographically smaller than $B''$ and hence $e \in \IP(B'') = S'' \smallDisjointUnion T''$.		

		We will prove the lemma by first relating $(S',T',A')$ and $(S'',T'',A'')$ and then applying the appropriate case of \eqref{eq:decomps}.
		Note that, as $B' \intOrder B''$, Theorem \ref{theorem:internalOrderIsAPoset} implies that $\IP(B') \subsetneq \IP(B'')$ and $A'' \subseteq A'$.
		
		To prove the first case of the proposition, let $e \in S''$.
		Then $S' = S'' - \{e\}$ and the perfection of $\matroid$ implies that $T' = T'' - T_{e}$ and $A' = A'' \cup T_{e}$.
		We now apply one of the two last cases of \eqref{eq:decomps} to obtain
			\begin{alignat*}{2}
				S &= S'' - e &&= S',
				\\ T &= T'' &&= T' \cup T_{e}, \text{ and}
				\\ A &= A'' &&= A' - T_{e} - a',
			\end{alignat*}
		proving the lemma in the case when $e \in S''$.

		Now suppose $e \in T''$.
		The prefection of $\matroid$ implies that there is a unique $f \in S''$ such that $e \in T_{f}$.
		It follows that in this case $S' = S''$, $T'= T'' - \{t \in T_{f} \suchthat t \le e\}$, and $A' = A'' \cup \{t \in T_{f} \suchthat t < e\}$.
		Using these facts and the second case of \eqref{eq:decomps}, we obtain $S = S'' = S'$, 
			\begin{alignat*}{2}
				T &= T'' - e &&=T' \cup \{t \in T_{f} \suchthat t \le e\} \text{ and}
				\\A &= A'' &&= A' - \{t \in T_{f} \suchthat t < e\},
			\end{alignat*}
		as desired.
	\end{proof}



	\begin{theorem}
	\label{theorem:perfectClosedUnderContractions}
		The class of perfect ordered matroids is closed under contraction.
	\end{theorem}
	\begin{proof}
		Let $B = S^{T}_{A} \in \bases(\cN)$.
		Then $B' = S'^{T'}_{A'} := B \cup e \in \bases(\matroid)$ is perfect and so 
			\[B' = \bigvee_{f \in S'}\minBasis_{\matroid}(f \cup T)\]
		and this is the unique way to write $B'$ as the join of principal bases.
		We now consider separately the three cases determined by which set of the $\STA$-decomposition of $B'$ contains $e$.
		
		If $e \in A'$, then $e \in A(\minBasis_{\matroid}(f \cup T'))$ for all $f \in S'$.
		It follows immediately that $\minBasis_{\cN}(f \cup T') = \minBasis_{\matroid}(f\cup T')-e$.
		Thus the first case of Lemma \ref{proposition:relateSTAwe} applies and we obtain $S = S'$ and $T(B;f) = T(B',f)$ for all $f \in S$.
		So
			\[T = T'
				= \disjointUnion_{f \in S'}T(B';f)
				=\disjointUnion_{f \in S}T(B;f)		\]
		shows that $B$ is perfect when $e \in A'$.

		Now suppose $e \in T'$.
		Then applying the second case of Lemma \ref{proposition:relateSTAwe} yields $S = S'$, $T = T'-e$ and $A = A'$.
		By the perfection of $\matroid$, there is a unique $g \in S'$ such that $e \in T(B';g)$.
		Note that $T(B;g) = T(B';g)-e$ and, for all $f \in S'-g$ the element $e$ is active in $\minBasis_{\matroid}(f \cup T')$.
		It follows that $T(B;f)=T(B';f)$ for all~$f \in S'-g$ and so
			\begin{align*}
				T &= T' - e
				\\&= \disjointUnion_{f \in S'}T(B';f)
				\\&= (T(B;g)-e) \smallDisjointUnion \disjointUnion_{f \in S'-g}T(B;g)
				\\&=\disjointUnion_{f \in S}T(B;g),				
			\end{align*}
		and so $B$ is perfect when $e \in T'$.

		Suppose $e \in S'$.
		We deal with the cases $T(B';e)=\emptyset$ and $T(B';e)\neq\emptyset$ separately.
		If $T(B';e) = \emptyset$, then $S = S'-e$ and $T = T'$.
		Moreover, for all $f \in S'-e$ we have 
			\begin{align*}
				T(\minBasis_{\matroid}(f \cup T')) &= T(\minBasis_{\matroid}(f \cup T))
				\\&= T(\minBasis_{\cN}(f \cup T)),				
			\end{align*}
		where the last equation holds by the first case of Lemma \ref{proposition:relateSTAnoe}.
		Now we see that
			\begin{align*}
				T = T'&= \disjointUnion_{f \in S'}T(B';f)
				\\&= \disjointUnion_{f \in S'-e}T(B';f)
				\\&= \disjointUnion_{f \in S}T(B;f)
			\end{align*}
		as desired.

		Finally suppose $e \in S'$ and that $T(B';e)\neq \emptyset$.
		Then $S = S' - e \cup a$ and $T = T'-a$.
		The perfection of $B'$ implies that $a \in T(B;f)$ \Iff $f=e$.
		For any $f \in S'$, write $B'_{f} = \minBasis_{\matroid}(f \cup T')$.
		When $f=e$, we have $a \in T(B';f) = T(B'_{f};f)$ and~$T(B';f)\in \independents(\cN)$.
		So in this case we have 
			\begin{align*}
			S(\minBasis_{\cN}(T(B';f))) &= a, and
			\\ T(\minBasis_{\cN}(T(B';f)))&=T(B';f)-a.
			\end{align*}
		When $f \in S'-e$, let $B''_{f} := \minBasis_{\matroid}(f \cup T(B';f) \cup e)$.
		Applying the first case of Lemma \ref{proposition:relateSTAnoe} gives 
			\begin{alignat*}{2}
				S(\minBasis_{\cN}(f \cup T(B';f))) &= S(B'_{f}) &&= \{f\},
				\\T(\minBasis_{\cN}(f \cup T(B';f))) &= T(B'_{f}) &&= T(B';f), and
				\\A(\minBasis_{\cN}(f \cup T(B';f))) &= A(B'_{f})-a'_{f} &&
			\end{alignat*}
		where $a'_{f}$ is the maximal element of $C(B'_{f};e)-(F \cup T(B;f) \cup e)$.
		We now have all of the ingredients necessary to show that $B$ is perfect when $e \in S'$:
			\begin{align*}
				T &= T' - a 
				\\&= \disjointUnion_{f \in S'} T(B';f) - a
				\\&= (T(B';e)-a) \smallDisjointUnion \disjointUnion_{f \in S'-e} T(B';f)
				\\&= T(B;a) \smallDisjointUnion \disjointUnion_{f \in S-a} T(B;f)
				\\&=\disjointUnion_{f \in S}T(B;f).
			\end{align*}
	\end{proof}

Combining Corollary \ref{corollary:delIsPerfect} and Theorem \ref{theorem:perfectClosedUnderContractions} proves Theorem \ref{theorem:minorClosedIfDelOffB0}.
We saw in Example \ref{example:notMinorClosed} that if we wish to preserve the ordering of the ground set of a perfect matroid~$\matroid$ when passing to a minor $\cN$, then in general we may only delete or contract elements of $B_{0}$ that are coloops.
If, on the other hand, we allow a reordering of the ground set of $\cN$, then we have found no example of a minor of a perfect matroid that is not perfect.
Thus we close by offering the following conjecture that would follow directly from our results here and Conjecture \ref{conjecture:unorderedDelAlwaysPerfect}.

\begin{conjecture}
\label{conjecture:minorClosed}
	The family of internally perfect unordered matroids is minor-closed.
\end{conjecture}


\section{Internally Perfect Matroids and Stanley's Conjecture}
\label{section:internallyPerfectMatroids}

	As mentioned in the introduction, our study of internally perfect matroids is motivated by the desire to prove Stanley's Conjecture using the internal order of a matroid as follows:
	Let $\matroid = (E, \bases, \phi)$ be an ordered matroid and let $\poset$ be the internal order of $\matroid$ less the artificial top element $\hat{1}$.
	Write $F = E - (B_{0}(\matroid) \cup \cL(\matroid)$ and let $\cS$ be the monoid over $\bbN$ generated by $F$, that is,
		\[\cS := \bigoplus_{f \in F} \bbN \bfe_{f}.\]
	The dominance relation on $\cS$ is defined by$\bfu \preccurlyeq_{\dom} \bfv$ if $\bfv_{f}-\bfu_{f}\ge 0$ for all $f \in F$.
	Call a map $\mu: \poset \to (\cS, \preccurlyeq_{\dom})$ \defn{valid} if $\mu$ is a height-preserving map whose image is an order ideal of $(\cS, \preccurlyeq_{\dom})$.
	Then since $\poset$ is pure and the number of bases of $\matroid$ at height $i$ in $\poset$ equals $h_{i}(\matroid)$, the existence of a valid map $\mu$ implies that $\matroid$ satisfies Stanley's Conjecture.

	The main result of this article is that internally perfect matroids satisfy Stanley's Conjecture:
	\begin{theorem}
		\label{theorem:perfectImpliesStanley}
		Given an unordered matroid $\matroid$, if there exists an ordering of the ground set that makes $\matroid$ into an internally perfect matroid then $\matroid$ satisfies Stanley's Conjecture.
	\end{theorem}
	In order to prove that such a matroid satisfies Stanley's Conjecture we need to produce a pure order ideal whose $\cO$-sequence equals the $h$-vector of the matroid.
	We will prove the theorem via a sequence of lemmas that actually imply a stronger result: the internal order of an internally perfect matroid is isomorphic to a pure order ideal.

	For the remainder of this section we fix a rank-$r$ internally perfect matroid $\matroid$ on the ground set $E$.
	If $e \in E$ is a loop in $\matroid$, then $e$ does not occur in any basis.
	This in turn implies that the $h$-vector of $\matroid - \{e\}$ equals the $h$-vector of $\matroid$.
	Therefore, we assume without loss of generality that $\matroid$ contains no loops.
	The first step is to define the monoid in which we will construct the appropriate order ideal.
	Let $B_{0}$ be the lexicographically-least basis of $\matroid$ and write $F := E - B_{0}$.
	For each $f \in F$ fix a generator $\bfe_{f}$ and let 
		\[\cS := \bigoplus_{f \in F}\bbN \bfe_{f}.\]
	As $|F| = h_{1}(\matroid)$ we have $\cS \isomorphic \bbN^{h_{1}}$.

	Recall that, by Theorem \ref{theorem:internalOrderFacts}, there are exactly $h_{i}(\matroid)$ bases at height $i$ in the internal order of $\matroid$.
	Thus a natural next step is to define a map that sends the bases of $\matroid$ to elements of $\cS$ in such a way that the coordinate sum of the image of each basis equals the cardinality of its set of internally passive elements.
	Since the matroid $\matroid$ is internally perfect, Proposition \ref{proposition:joinsOfProjections} guarantees that each basis~$B = S^{T}_{A}$ of $\matroid$ can be written uniquely as the join of principal $f$-bases 
	\[\bigvee_{f \in S} \minBasis(f \cup T(B;f))\]
	where $S \subseteq E - B_{0}$ and the set $\{T(B;f)\suchthat f \in S\}$ is a partition of $T$.
	It is therefore natural to consider the map $\mu: \bases(\matroid) \to \cS$ defined by
		\[\mu(B) = \sum_{f \in S(B)} \big|\{f\} \cup T(B;f)\big|\bfe_{f}.\]
	It is evident that for any basis $B$ the coordinate sum of $\mu(B)$ is equal to the number of internally passive elements of $B$, and hence $\mu$ is height-preserving.

	The first lemma we will prove is that $\mu$ injective.
	\begin{lemma}
		\label{lemma:muIsInjective}
		Let $\matroid$ be an internally perfect matroid and let $\mu: \bases(\matroid) \to \cS$ be as above.
		Then $\mu$ is injective.
	\end{lemma}
	\begin{proof}
		Suppose there are two bases,~$B = S^{T}_{A} $ and~$ B' = S'^{T'}_{A'}$, of an internally perfect matroid $\matroid$ such that~$\mu(B) = \mu(B')$.
		Then by the definition of $\mu$ we have $S = S'$ and, for all $f \in S$, the cardinalities of the sets $T(B;f)$ and $T'(B';f)$ are equal.
		As the set of all principal $f$-bases forms a maximal chain in $\poset_{\Int}$, there can be at most one prinicpal $f$-basis at any height.
		This implies that 
		\[\minBasis(\{f\} \cup T(B;f)) = \minBasis(\{f\} \cup T'(B';f))\]
		for all $f \in F$.
		Since $\matroid$ is internally perfect, so are both $B$ and $B'$. 
		So applying Proposition \ref{proposition:joinsOfProjections} for perfect bases, we see that
			\begin{align*}
				B &= \bigvee_{f \in S} \minBasis(f \cup T(B;f))
				\\&= \bigvee_{f \in S'} \minBasis(f \cup T'(B';f))
				\\&= B',
			\end{align*}
		which shows $\mu$ is injective.
	\end{proof}

	We will now prove that if $\matroid$ is an internally perfect matroid, then $\matroid$ satisfies Stanley's Conjecture by showing that the image of $\mu$ is a pure order ideal in $\cS$. 
	Given a matroid $\matroid$, we call an order ideal $\cO \subset \cS(\matroid)$ \defn{valid} for $\matroid$ if $\cO$ is a pure order ideal whose $\cO$-sequence equals the $h$-vector of $\matroid$. 	

	\begin{theorem}
		\label{theorem:muGivesOrderIdeal}
		Let $\matroid$ be an internally perfect matroid and let $\mu: \bases(\matroid) \to \cS$ be as above.
		Then the image of $\mu$ is a valid order ideal for $\matroid$.
	\end{theorem}
	\begin{proof}
		To see that the image of $\mu$ is an order ideal we need to check that for any basis $B = S^{T}_{A}$ of $\matroid$ and any $\bfe_{f}$ in the support of $\mu(B)$, the vector $\mu(B)-\bfe_{f}$ is in the image of $\mu$.
		Note that for any element $f \in F$ the generator $\bfe_{f}$ is in the support of $\mu(B)$ \Iff the element $f \in S$.
		Since $B$ is a perfect basis it covers $|S|$ bases in the internal order by Proposition \ref{proposition:characterizationsByCovers}.
		By Lemma \ref{lemma:muIsInjective}, the map~$\mu$ is injective, so the $|S|$ bases covered by $B$ in $\poset_{\Int}$ get mapped to $|S|$ distinct vectors in $\cS$ with support contained in $S$ whose coordinate sum is one less than that of $\mu(B)$.
		As $\mu(B)$ has support $|S|$, it follows that for every $f \in S$ there is a unique basis $B' \in \bases$ such that $\mu(B') = \mu(B)-\bfe_{f}$.

		Now we check that the image of $\mu$ is a pure order ideal.
		By Theorem \ref{theorem:internalOrderFacts}, the internal order $\widehat{\poset}_{\Int}$ is a graded lattice.
		It follows that $\poset_{\Int}$ (the internal order with the top element removed) is a pure graded poset, and so the image of $\mu$ is a pure order ideal.

		Finally, since $\mu$ is an injective map sending the $h_{i}(\matroid)$ bases in $\poset_{\Int}$ at height~$i$ to vectors with coordinate sum $i$ in $\cS$, it follows that the $h$-vector of $\matroid$ is a pure~$\cO$-sequence completing the proof.
	\end{proof}

	The preceding proof shows that the internal order of a perfect matroid $\matroid$ is isomorphic to the corresponding valid order ideal given by the map $\mu$, and hence that $\matroid$ satisfies Stanley's Conjecture.
	Stanley's Conjecture is known to hold for a number of families of matroids, and in the next section we compare these families to the family of internally perfect matroids.
	In particular, in Example \ref{example:perfectNew} of the next section we provide a perfect matroid that is not in any of the families for which Stanley's Conjecture is known to hold.


\section{Perfect Matroids: Constructions, Examples and Counterexamples}
\label{section:examples}

We now turn to the construction of internally perfect matroids.
We have already seen in Proposition \ref{proposition:perfectBasesExist} that every rank-$2$ matroid is internally perfect.
Other matroids that are trivially internally perfect include the graphic matroid $\matroid(C_n)$ where $C_n$ is the the cycle on $n$ vertices, as well as the $0$- and $1$-sums of perfect matroids.

The smallest graphic matroid that is not internally perfect for any ordering of its ground set is $\matroid(K_4)$.
As $K_{4}$ is planar and self-dual, we see that the cographic matroids are not contained in the set of internally perfect matroids.
On the other hand we now construct an infinite family of nontrivial perfect cographic matroids as follows.
Let $\matroid_{1} = \matroid(C_{2})$ be the cycle matroid on two elements represented by the matrix $M_{1} := [1~1]$ and, for $r >1$, define $\matroid_{r}:= \PE(\PL(\matroid_{r-1},\emptyset), \{2r-2,2r-1\})$ where $\PE$ and $\PL$ denote the principal extension and principal lift, respectively.
It is easy to see that, for all positive integers $r$, the matroid $M_{r}$ is a graphic.  
Indeed, for $r>1$ if we let $G_{r} = (V_{r}, E_{r})$ be defined by $V_{r-1} \cup \{r+1\}$ and $E_{r} = E_{r-1} \cup \{e,f\}$ where $e = (r,r+1)$ and $f= (1,r+1)$, then we have $\matroid(G_{r}) = \matroid_{r}$.
For $r = 1,2,3$, the graphs corresponding to the $\matroid_{r}$ are given in Figure \ref{figure:graphExamples}.
The graph corresponding to $\matroid_{r}$ is planar and self-dual, so $\matroid_{r}$ is also cographic.

\begin{figure}[htbp]
\centering
	\begin{tikzpicture}[-,shorten >=1pt,node distance=1.75cm,
  thick,main node/.style={draw, fill=black, circle, inner sep=0pt}]

  \node[main node] (1) {1};
  \node[main node] (2) [right of=1] {2};
  \node[main node] (3) [right of=2] {3};
  \node[main node] (4) [right of=3] {4};
  \node[main node] (5) [right of=4] {5};
  \node[main node] (6) [right of=5] {6};
  \node[main node] (7) [right of=6] {7};
  \node[main node] (8) [right of=7] {8};
  \node[main node] (9) [right of=8] {9};

  \path[every node/.style={font=\sffamily\small}]
    (1) edge node {1} (2)
    	edge [bend right] node [right] {2} (2)
    (3) edge node {1} (4)
    	edge [bend right] node [right ]{2} (4)
    	edge [bend right] node {4} (5)
    (4) edge node {3} (5)
    (6) edge node {1} (7)
    	edge [bend right] node [right] {2} (7)
    	edge [bend right] node [right] {4} (8)
    	edge [bend right] node [right] {6} (9)
    (7) edge node  {3} (8)
    (8) edge node  {5} (9);
\end{tikzpicture}
\caption{graphs giving rise to $\matroid_{r}$ for $r=1,2,3$}
\label{figure:graphExamples}
\end{figure}

	\begin{proposition}
	\label{proposition:myGraphicsArePerfect}
		For any positive integer $r$, the matroid $M_{r}$ is internally perfect with respect to the natural ordering on its ground set.
	\end{proposition}
	\begin{proof}
		We proceed by induction on $r$ and note that the base case is trivial.
		For $r$ greater than $1$, each basis of $\matroid_{r}$ takes exactly one of the following three forms:
		\begin{enumerate}
            \item \label{eq:type1}$B = B' \cup \{2n-1\}$ where $B'$ is a basis of $\matroid_{r-1}$;
            \item \label{eq:type2}$B = B' \cup \{2n\}$ where $B'$ is a basis of $\matroid_{r-1}$;
            \item \label{eq:type3}$B = B' - \{2n-2\} \cup \{2n-1,2n\}$ where $B'$ is a basis of $\matroid_{r-1}$ such that $2n-2 \in B'$.
        \end{enumerate}
        If $B = B' \cup e$ is a basis of $\matroid_{r}$ of type \eqref{eq:type1} or \eqref{eq:type2} then $e$ is internally active in $B$ and hence $S(B)=S(B')$ and $T(B)=T(B')$.
        It follows that $B$ is perfect.
        Otherwise we have $B = B' - \{2n-2\} \cup \{2n-1,2n\}$ where both $e:= 2n-1$ and $f:=2n$ are internally passive in $B$.
        Moreover $e$ is in the $f$-part of $T$ and in no other $g$-part of~$T$ for $g \in S(B)$.
        It follows that $T = \smallDisjointUnion_{e \in S(B)} T(B;e)$ and so $B$ is internally perfect.
        The result follows.
	\end{proof}

	\begin{example}
  \label{example:perfectNew}
  To illustrate this proposition we give the internal order of $\matroid_{4}$ in Figure~\ref{figure:intOrderGoodExample}.
    The internal order of $\matroid_{3}$ is a subposet of $\poset_{\Int}(\matroid_{4})$ and we have highlighted the corresponding bases in blue.
    Notice that these bases are precisely those containing the element $7$ but not the element $8$, and in every such basis the element~$7$ is an internally active element. \qed
  \end{example}
	\begin{figure}[htbp]
	\centering
	\tikzstyle{every node}=[draw=black, thick, fill = white, rectangle,inner sep=2pt]
	\begin{tikzpicture}[scale=1, vertices/.style={draw, fill=black, circle, inner sep=0pt}]
               \node (0)  at (-0   +0,0) 		    {\scriptsize$\color{blue}{\emptyset^{}_{1357}}$};
               \node (13) at (-15/8+ 0/4,8/5) 	{\scriptsize$\color{blue}{{2}^{}_{357}}$};
               \node (5)  at (-15/8+ 5/4,8/5) 	{\scriptsize$\color{blue}{{4}^{}_{157}}$};
               \node (2)  at (-15/8+10/4,8/5)	  {\scriptsize$\color{blue}{{6}^{}_{137}}$};
               \node (1)  at (-15/8+15/4,8/5) 	{\scriptsize${{8}^{}_{135}}$};
               \node (18) at (-5   + 0/4,16/5) 	{\scriptsize$\color{blue}{{24}^{}_{57}}$};
               \node (15) at (-5   + 5/4,16/5) 	{\scriptsize$\color{blue}{{26}^{}_{37}}$};
               \node (26) at (-5   +10/4,16/5) 	{\scriptsize$\color{blue}{{4}^{3}_{57}}$};
               \node (7)  at (-5   +15/4,16/5) 	{\scriptsize$\color{blue}{{46}^{}_{17}}$};
               \node (10) at (-5   +20/4,16/5) 	{\scriptsize$\color{blue}{{6}^{5}_{17}}$};
               \node (14) at (-5   +25/4,16/5) 	{\scriptsize${{28}^{}_{35}}$};
               \node (6)  at (-5   +30/4,16/5) 	{\scriptsize${{48}^{}_{15}}$};
               \node (3)  at (-5  +35/4,16/5) 	{\scriptsize${{68}^{}_{13}}$};
               \node (4)  at (-5  +40/4,16/5) 	{\scriptsize${{8}^{7}_{13}}$};
               \node (20) at (-55/8+0/4,24/5) 	{\scriptsize$\color{blue}{{246}^{}_{7}}$};
               \node (23) at (-55/8+5/4,24/5) 	{\scriptsize$\color{blue}{{26}^{5}_{7}}$};
               \node (28) at (-55/8+10/4,24/5) 	{\scriptsize$\color{blue}{{46}^{3}_{7}}$};
               \node (31) at (-55/8+15/4,24/5) 	{\scriptsize$\color{blue}{{6}^{35}_{7}}$};
               \node (19) at (-55/8+20/4,24/5) 	{\scriptsize${{248}^{}_{5}}$};
               \node (16) at (-55/8+25/4,24/5) 	{\scriptsize${{268}^{}_{3}}$};
               \node (17) at (-55/8+30/4,24/5) 	{\scriptsize${{28}^{7}_{3}}$};
               \node (8)  at (-55/8+35/4,24/5) 	{\scriptsize${{468}^{}_{1}}$};
               \node (27) at (-55/8+40/4,24/5) 	{\scriptsize${{48}^{3}_{5}}$};
               \node (9)  at (-55/8+45/4,24/5) 	{\scriptsize${{48}^{7}_{1}}$};
               \node (11) at (-55/8+50/4,24/5) 	{\scriptsize${{68}^{5}_{1}}$};
               \node (12) at (-55/8+55/4,24/5) 	{\scriptsize${{8}^{57}_{1}}$};
               \node (21) at (-35/8+0/4,32/5) 	{\scriptsize${{2468}^{}_{}}$};
               \node (22) at (-35/8+5/4,32/5) 	{\scriptsize${{248}^{7}_{}}$};
               \node (24) at (-35/8+10/4,32/5) 	{\scriptsize${{268}^{5}_{}}$};
               \node (25) at (-35/8+15/4,32/5) 	{\scriptsize${{28}^{57}_{}}$};
               \node (29) at (-35/8+20/4,32/5) 	{\scriptsize${{468}^{3}_{}}$};
               \node (30) at (-35/8+25/4,32/5) 	{\scriptsize${{48}^{37}_{}}$};
               \node (32) at (-35/8+30/4,32/5) 	{\scriptsize${{68}^{35}_{}}$};
               \node (33) at (-35/8+35/4,32/5) 	{\scriptsize${{8}^{357}_{}}$};
       \foreach \to/\from in {0/1, 1/3, 1/4, 1/14, 1/6, 2/3, 3/8, 3/16, 3/11, 4/12, 4/9, 4/17,  5/6, 6/8, 6/9, 6/27, 6/19,  7/8, 8/21, 8/29, 9/22, 9/30,  10/11, 11/24, 11/32, 12/25, 12/33, 13/14,  14/16, 14/17, 14/19,  15/16, 16/24, 16/21, 17/25, 17/22, 18/19, 19/21, 19/22, 20/21, 23/24, 26/27, 27/29, 27/30, 28/29, 31/32}
       \draw [-] (\to)--(\from);
       18)15)26)7) 10)
       20 23 28 31
       \foreach \to/\from in {0/2, 0/13, 0/5, 2/10, 2/15, 2/7, 5/26, 5/18, 5/7, 7/20, 7/28, 10/23, 10/31, 13/18, 13/15, 15/20, 15/23, 18/20, 26/28}
       \draw [blue, -] (\to)--(\from);
       \end{tikzpicture}
       \caption{the internal order of $\matroid_{4}$}
       \label{figure:intOrderGoodExample}
	\end{figure}

  Let $\matroid$ be a rank-$r$ matroid on $n$ elements and let $h(\matroid) = (h_{0},h_{1},\dots, h_{s},\dots, h_{r})$ where $h_{s}$ is the last nonzero entry of $h(\matroid)$.
  Then $\matroid$ satisfies Stanley's Conjecture if any of the following hold:
      \begin{enumerate}[]      
      \item \label{property:1} $\matroid^{*}$ is graphic \cite{lopez1997chip},
      \item \label{property:2} $\matroid^{*}$ is transversal \cite{oh2013generalized},      
      \item \label{property:3} $\matroid^{*}$ has no more than $n-r+2$ parallel classes \cite{constantinescu2012generic}, 
      \item \label{property:4} $n\le9$ or $\rank(\matroid^{*}) \le 2$ \cite{deloera2011h},
      \item \label{property:5} $r \le 4$ \cite{klee2014lexicographic},
      \item \label{property:6} $\matroid$ is paving \cite{merino2012structure},
      \item \label{property:7} $\matroid$ is  a truncation \cite{constantinescu2012generic}, or
      \item \label{property:8} $h_{s}\le 5$ \cite{constantinescu2012generic}.      
      \end{enumerate}
  This list represents the state of the art concerning Stanley's Conjecture at the writing of this article.
  For the sake of brevity let us call a matroid $ \matroid $ \defn{interesting} if it satisfies none of the above properties.
  \emph{A priori} it is not clear that there are any interesting internally perfect matroids.
  Thus, to show that Theorem \ref{theorem:perfectImpliesStanley} is of theoretical interest, we must exhibit such a matroid.
  In the next example we describe an interesting, perfect, rank-7 matroid on 10 elements.
  Then we generalize this matroid to obtain an infinite family of interesting matroids and we conjecture that every such matroid is perfect.

  \begin{example}
  \label{example:interestingPerfectMatroid}
  Let  $\cN = \cN(N)$ be the rank-3 ordered vector matroid over $\bbQ$ on 10 elements given by the columns of the matrix      
      \[ N :=
        \makeatletter\c@MaxMatrixCols=12\makeatother
          \kbordermatrix{          
          \empty & e_{1} & e_{2} & e_{3} & e_{4}  & e_{5} & e_{6} & e_{7} & e_{8}& e_{9} & e_{10} \cr
          \empty & 2 & 1 & 3 & 3 & -1 & -1 & 0  & 0  & -1 & -1\cr
          \empty & 1 & 1 & 1 & 1 & 1  &  1 & 1  & 1  & 0  & 0\cr
          \empty & 0 & 0 & 0 & 0 & 0  & 0  & -1 & -1 & 1  & 1\\
          }      
      \]      
    and let $\matroid = \cN^{*}$ be the dual of $\cN$.
  
    We first show that $\matroid$ is interesting and then prove that $\matroid$ is internally perfect.
    The restriction of $\matroid^{*} = \cN$ to $\{e_{1},e_{2},e_{3},e_{5}\}$ is isomorphic to the uniform matroid~$\uniformMatroid{2}{4}$ which shows that $\matroid^{*}$ is not graphic.
    So $\matroid$ does not satisfy Property~\eqref{property:1} above.
    To see that~$\matroid$ doesn't satisfy \eqref{property:2} above note that the restriction of $\matroid^{*}$ to the set~
    $\{e_{5},e_{6},e_{7},e_{8},e_{9},e_{10}\}$ is the matroid obtained from the cycle $C_{3}$ by adding a parallel element to each edge.
    This restriction is not transversal by Theorem 14.3.1 in~\cite{welsh1976matroid}.  
    It is easy to see that transversal matroids are closed under taking restrictions, and it follows that $\matroid^{*}$ is not transversal.
    Also, $\matroid^{*}$ is a rank-three matroid on ten elements with six parallel classes, and so satisfies neither \eqref{property:3} nor \eqref{property:4} above.
  
    Now we show $\matroid$ does not satisfy any of the final four properties listed above.
    As the rank of  $\matroid$ is six, $\matroid$ violates \eqref{property:5}.
    The circuits of $\matroid$ are 
    \[
    \circuits(\matroid) = \{1234,125678,1345678,2345678,12569\overline{10},134569\overline{10},234569\overline{10},789\overline{10}\}
    \]
    where, for example, $789\overline{10}$ refers to the set $\{e_{7},e_{8},e_{9},e_{10}\}$.
    Since $1234$ is a circuit with fewer than $\rank(\matroid)$ elements, $\matroid $ is not paving.
    Moreover, one can show that for any circuit $C$ of $\matroid$ the cardinality of the closure of $C$ is at most 8. 
    So $\matroid$ contains no spanning circuits.
    By Remark 1.12 in \cite{constantinescu2012generic}, this shows that $\matroid$ is not a truncation.
    Finally, the $h$-vector of $\matroid$ is $(1, 3, 6, 10, 13, 15, 14, 6)$ which shows that~$\matroid$ does not satisfy \eqref{property:8} above.
    So $\matroid$ is interesting.

    Now we show that $\matroid$ is internally perfect \wrt the linear order induced by the order of the columns of the matrix $N$.
    By Proposition \ref{proposition:perfectionFiltersDown}, it is enough to check that the maximal bases of $\matroid$ (with respect to the internal order) are perfect.
    Each of these bases is of the form $S^{T}_{A}$ where $S=\{3,7,9\}$ and $A=\emptyset$.
    We give these bases below together with the partition $3^{T(B_{i};3)}7^{T(B_{i};7)}9^{T(B_{i};9)}$:
      \begin{alignat*}{4}
        B_{1} &= {379}^{1256}  &&= 3^{12}7^{56}9^{\emptyset}\\ 
        B_{2} &= {379}^{1258}  &&= 3^{12}7^{\emptyset}9^{58}\\  
        B_{3} &= {379}^{1456}  &&= 3^{\emptyset}7^{1456}9^{\emptyset}\\
        B_{4} &= {379}^{1458}  &&= 3^{\emptyset}7^{\emptyset}9^{1458}\\ 
        B_{5} &= {379}^{2456}  &&= 3^{2}7^{456}9^{\emptyset}\\  
        B_{6} &= {379}^{2458}  &&= 3^{2}7^{\emptyset}9^{458}.
      \end{alignat*}
    For $i \in [6]$, we have $T(B_{i}) = \disjointUnion_{f \in S(B_{i})} T(B_{i};f)$ and so $B_{i}$ is internally perfect.
    It follows that $\matroid$ is internally perfect, as desired. \qed
  \end{example}

  Building on the previous example, we now define an infinite family of interesting matroids.
  For $n \ge 3$, let $G = G_{n,D}$ be a graph consisting of the double cycle $C^{2}_{n}$ (the $n$-cycle with two copies of each edge) together with some subset $D$ of the edges~$\{(1,i) \suchthat i \in [n] \setminus{1}\}$.
  Orient each edge $\{i,j\}$ so that $i \to j$ if $i<j$, and order the set of all edges lexicographically.
  Write $N = N(n,D)$ for the vertex-edge incidence matrix of $G$ and let $N'$ be the $(n \times 4)$-matrix with first two rows given by
    \[
    \begin{pmatrix}
      2 & 1 & 3 & 3\\
      1 & 1 & 1 & 1
    \end{pmatrix}
    \]
  and all other rows having all entries equal to zero.
  Let $\cN = \cN_{n,D}$ be the ordered vector matroid given by the matrix $[N' | N]$ and let $\matroid = \cN^{*}$.
  Note that $\cN$ is a rank-$n$ matroid on $2n + |D| + 4$ elements with $n+|D|+3$ parallel classes.
  Moreover, the matroids $\uniformMatroid{2}{4}$ and $C^{2}_{n}$ can be obtained as restrictions of $\cN$.
  This shows that the matroid~$\matroid = \cN^{*}$ does not satisfy any of the properties \eqref{property:1}-\eqref{property:4} above.
  The matroid~$\matroid$ has rank $n+|D|+4 \ge 6$ and the set $\{e_{1},e_{2},e_{3}\}$ is always a circuit of $\matroid$, and so~$\matroid$ does not satisfy \eqref{property:5} or \eqref{property:6}.
  Moreover, it is not hard to show by induction on $n$ and $|D|$ that $\matroid$ has no spanning circuit and that the last entry of the $h$-vector of $\matroid$ is at least 6.
  Thus $\matroid$ does not satisfy either \eqref{property:7} or~\eqref{property:8}.

  Computer experiments have shown that for $n \le 7$ and any collection of diagonals~$D$ the matroid $\matroid = (\cN_{n,D})^{*}$ is internally perfect. 
  We conjecture that this is always the case.
    \begin{conjecture}
    \label{conjecture:newPerfects}
      For $n \ge 3$, let $\matroid$ be the dual of the ordered matroid $\cN_{n,D}$ defined above. 
      Then $\matroid$ is internally perfect.
    \end{conjecture}

We now compare the family of internally perfect matroids to each of the families of matroids for which Stanley's Conjecture is known to hold.
We write $\intPerfectMatroids$ for the set of all internally perfect matroids and $\cF_{i}$ for the family of all matroids satisfying property $(i)$ above.
The interesting perfect matroid in Example \ref{example:interestingPerfectMatroid} shows that $\intPerfectMatroids \nsubseteq \cF_{i}$ for all $i \in [8]$.
We will now show that none of the opposite inclusions hold by providing a matroid in each family $\cF_{i}$ that is not internally perfect for any linear order of its ground set.

\begin{example}
\label{example:K4notPerfect}
As noted at the beginning of this section the matroid~$\matroid = \matroid(K_{4})$ is a cographic matroid that is not internally perfect for any linear ordering of the ground set.
One easily verifies that $\matroid$ is a rank-$3$, self-dual, paving matroid on six elements.
Moreover, it satisfies the first condition of property \eqref{property:4}.
It follows that~$\cF_{i} \nsubseteq \intPerfectMatroids$ for $i \in \{1,4,5,6\}$.
\qed
\end{example}

\begin{example}
\label{example:UrnNotPerfect}
Next we consider the rank-$r$ uniform matroid on $n$ elements, $\matroid := \uniformMatroid{r}{n}$.
We will show that if $r>2$ and $n\ge r+1$, that $\matroid$ is not internally perfect with respect to the natural ordering of the ground set, from which it will follow that $\matroid$ is not perfect for any such ordering.
Note that, since $\matroid^{*}$ is also uniform, if $B$ is a basis of $\matroid$ and $e \in B$ then the fundamental cocircuit $C^{*}(B;e) = E - B \cup e$.
It follows that if $B$ and $B'$ are two bases such that $B'= B - e \cup f$ and if $e'\in B \cap B'$, then 
  \begin{equation}
  \label{equation:circuitEq}
    C^{*}(B;e')=C^{*}(B';e')-e \cup f.
  \end{equation}
Now let $B$ be an $f$-principal basis of $\matroid$ and let $g \in [n]-[r]$. 
The basis $B' := B-f \cup g$ is a $g$-principal basis.
Moreover, it is evident from \eqref{equation:circuitEq} that if $e \in B \cap B'$ is internally passive in $B$ if and only if it is internally passive in $B'$.
Thus if we write $B = f^{T}_{A}$, then $B' = g^{T}_{A}$.
It follows that the join of $B$ and $B'$ in $\poset_{\Int}$ is perfect if and only if $T = \emptyset$.
This implies the following result. 

  \begin{proposition}
  \label{proposition:perfectionAndUniform}
  The rank-$r$ uniform matroid on $n$ elements is perfect if and only if $r=2$ or $n = r+1$.
  \end{proposition}

In particular, the rank-$3$ uniform matroid $\uniformMatroid{3}{5}$ is not perfect.
It is well-known (and easy to prove) that $(\uniformMatroid{3}{5})^{*} = \uniformMatroid{2}{5}$ is a rank-$2$ transversal matroid.
Moreover, $\uniformMatroid{3}{5}$ is obviously a truncation of $\uniformMatroid{4}{5}$.
Thus $\uniformMatroid{3}{5} \in \cF_{i}$ for $i \in \{2,4,5,6,7\}$ and we have shown that $\cF_{i} \nsubseteq \intPerfectMatroids$  for any such $i$.
\qed
\end{example}


\begin{example}
\label{example:r2n8}
Let $\cN$ be the rank-2 linear matroid on eight elements defined by the matrix 
  \[
    \begin{pmatrix}
      1 & 1 & 0 & 0 &1  &1  & 1   & 1  \\
      0 & 0 & 1 & 1 &-1 &-1  & 1 & 1   
    \end{pmatrix}
  \]
and let $\matroid = \cN^{*}$.
Then $\matroid$ is a rank-6 matroid on eight elements whose dual has four parallel classes.
Moreover, the $h$-vector of $\matroid$ is $\{1,2,3,4,5,6,3\}$.
It follows that $\matroid \in \cF_{i}$ for $i \in \{3,4,8\}$c .
Checking all $8!$ linear orders of the ground set, one can show that $\matroid$ is not internally perfect for any such ordering.
This implies that $\intPerfectMatroids \nsubseteq \cF_{i}$ for $i \in \{3,4,8\}$.
\qed
\end{example}

The previous examples show that none of the families of matroids for which Stanley's Conjecture is known to hold consist entirely of perfect matroids.
Moreover, note that the duals of the matroids $\uniformMatroid{3}{5}$ and $\matroid$ from Example \ref{example:UrnNotPerfect} and  \ref{example:r2n8}, respectively, are rank-$2$ matroids.
This implies that $(\uniformMatroid{3}{5})^{*}$ and $\matroid^{*}$ are internally perfect for any ordering of their respective ground sets, and hence that the set of internally perfect matroids is not closed under matroid duality.

\bibliographystyle{plain}
\bibliography{internalOrderStanley}	
\end{document}